\documentclass[reqno,11pt]{amsart}
\usepackage{amsmath,amsthm,amssymb,amsfonts,color}

\newtheorem{theorem}{Theorem}[section]
\newtheorem{corollary}[theorem]{Corollary}
\newtheorem{lemma}[theorem]{Lemma}

\newtheorem{remark}[theorem]{Remark}
\newtheorem{definition}[theorem]{Definition}

\newtheorem{proposition}[theorem]{Proposition}
\newtheorem{example}[theorem]{Example}

\numberwithin{equation}{section}

\def\H{\mathcal{H}}
\def\R{\mathbb R}
\def\N{\mathbb N}

\def\F{\mathcal F_{\a,Q}}
\def\G{\mathcal G_{\a,Q}}
\def\I{\mathcal I}
\def\Q{\mathcal{I}_\alpha}
\def\Il{{\mathcal I_{\rm log}}}
\def\lo{{\rm log}}
\def\P{P}
\def\M{\mathcal M}
\def\Z{\mathbb Z}
\def\eps{\varepsilon}
\def\e{\varepsilon}
\def\a{\alpha}

\def\vphi{\varphi}
\def\p{\varphi}
\renewcommand{\phi}{\varphi}
\def\pa{\partial}
\def\g{\gamma}
\def\O{\Omega}
\def\ov{\overline}

\def\bal{\begin{aligned}}
\def\eal{\end{aligned}}
\def\Om{\Omega}

\def\spt{\mathrm{spt}}
\def\Hd{\H^d}
\def\Hdm{\H^{d-1}}
\def\dP{\delta\!\P}
\newcommand\LinfE[1]{{\|#1\|}_{\stackrel{}{L^{\infty}(\partial E)}}}
\newcommand\LinfB[1]{{\|#1\|}_{\stackrel{}{L^{\infty}(\partial B)}}}

\newcommand\WinfB[1]{{\|#1\|}_{\stackrel{}{W^{1,\infty}(\partial B)}}}
\def\noi{}
\def\Kd{\mathcal{K}_\delta}
\def\Kdco{\mathcal{K}_\delta^{co}}
\def\diam{\mathrm{diam}}
\newcommand\WinfBd[1]{{\|#1\|}_{\stackrel{}{W^{1,\infty}(\partial B_{1/2})}}}
\def\ovE{\overline{E}}

\def\lt{\left}
\def\rt{\right}
\newcommand{\res}{\mathop{\hbox{\vrule height 7pt width.5pt depth 0pt\vrule height .5pt width 6pt depth
0pt}}\nolimits}

\def\ds{\displaystyle}
\title[Charged drops]{Existence and stability for a non-local isoperimetric model of charged liquid drops}

\author{Michael Goldman}
\address{Max-Planck-Institut f\"ur Mathematik,
Inselstrasse 22, 04103 Leipzig, Germany}
\email{goldman@mis.mpg.de}

\author{Matteo Novaga}
\address{Department of Mathematics, University of Pisa,
Largo Bruno Pontecorvo 5, 56127 Pisa, Italy}
\email{novaga@dm.unipi.it}

\author{Berardo Ruffini}
\address{Scuola Normale Superiore di Pisa, Piazza dei Cavalieri 4, 56126 Pisa, Italy}
\email{berardo.ruffini@sns.it}

\begin{document}

\begin{abstract}
We consider a variational problem related to the shape of charged liquid drops at equilibrium. 
We show that this problem never admits local minimizers with respect to $L^1$ perturbations preserving the volume. 
 However, we prove that the ball is stable under small $C^{1,1}$ perturbations when the charge is small enough. 
\end{abstract}

\maketitle


\section{Introduction}\label{droplets1}

\noi In this paper we study an isoperimetric variational problem where the perimeter, which is local and attractive, competes with the Riesz potential energy,  
which is non-local and repulsive. More precisely, we denote   
\[
\Q(E):=\inf\left\{\int_{\R^d\times \R^d}\frac{d\mu(x)d\mu(y)}{|x-y|^\a}:\,\mu(E)=1\right\},
\]
where $\a\in(0,d)$ and $E$ is a compact subset of $\R^d$ and consider the functional 

\begin{equation}\label{funcintro}
\F(E)\,:=\,P(E)\, +\, Q^2 \Q(E)
\end{equation}
where $Q>0$  is a parameter and where $P(E)$ denotes the perimeter of $E$ (which corresponds to $\H^{d-1}(\partial E)$ if $E$ has smooth
boundary, see \cite{amfupa}). We are in particular interested in the questions of existence and characterization of stables sets under volume preserving perturbations.
It turns out that the answer to these questions depends crucially on the  regularity of the allowed perturbations. In fact, we prove that on the one hand,
 there are no local (or global) minimizers of \eqref{funcintro} under volume constraint in the $L^1$ or even Hausdorff topology. This implies that there are no 
sets which are stables under such  perturbations. On the other hand, we prove that for small enough charge $Q$, the ball is stable under small $C^{1,1}$ perturbations.
 This comes as a by-product of the global minimality of such a ball in the class of ``regular enough'' 
 sets. \\

 \subsection{Description of the model}
For $\a=d-2$, $\Q(E)$ corresponds to the Coulombic interaction energy and the functional \eqref{funcintro} can be 
thought as modeling the equilibrium shape of a charged droplet for which surface tension and electric forces compete. Such charged droplets have received considerable attention since 
the seminal work of  Lord Rayleigh \cite{rayleigh} and are by now widely used in applications such as electrospray
ionization, fuel injection and ink jet printing. 
Starting with the pioneering experiments of Zeleny \cite{zeleny}, the following scenario  emerged. For small
charge, a spheric drop remains stable but when the charge 
overcomes a critical threshold $Q_c$,
which depends on the volume of the drop and on the characteristic constants of the liquid (surface tension and 
dielectric constant), a symmetry breaking occurs. Typically, 
the drop deforms and quickly develops conical shaped singularities, ejecting a very thin liquid jet
\cite{taylor,delamora,fontfried}.
This jet carries very little mass but a large portion of the charge.
 This type of behavior has been since then observed in more details and in various experimental setups (see for instance \cite{AMDL,DAMHL}).
 We emphasize in particular on \cite{doyle,abbas}, where the disintegration of an evaporating drop is observed, since a model very similar to \eqref{funcintro}
 has been proposed in \cite{roth,castle} to explain these experiments. We should stress the fact that the study of the unstable regime, which is still very poorly 
understood both experimentally and mathematically (see for instance \cite{miksis,fontfried,fuscojulin}),
 is far outside the scope of this paper. We focus instead on the rather simple variational model \eqref{funcintro} which hopefully captures, 
at least for small charges, most of the characteristics of the system. However, the unconditional (in term of $Q$)
 non-linear instability of the ball that we obtain in contrast with numerical and experimental 
observations, indicates that something is still missing in this model. 
A challenging question is identifying the relevant physical effect which stabilizes a charged drop.

In some applications such as electrowetting \cite{mugele} it is more natural to impose  the {\em electric potential} 
$V_0$ (see Definition \ref{potential})  on the boundary of $E$
 rather than the total charge $Q$.  In that case the energy of a drop $E$ takes the form 
\begin{equation}\label{funcpot}
 P(E)-V_0^2C_2(E),
\end{equation}
where for a set $E\subset \R^d$ with $d\ge 3$,
\[
C_{2}(E):= \min\left\{ \int_{\R^d}|\nabla u|^2\,dx: u\in H^1_0(\R^d),\,\text{$u\ge1$ on $E$}\right\}, 
\]
is the  capacitary functional. Notice that since $\I_{d-2}(E)=C_2(E)^{-1}$ for compact sets (see Remark \ref{rmkcap}), 
the functionals \eqref{funcintro} and \eqref{funcpot} are qualitatively similar. The analogy is in fact deeper since both functionals give rise to the same Euler-Lagrange equation.

\subsection{Main results of the paper}
The first main result of the paper is that, when $\a\in(0,d-1)$, for every given charge and
volume, quite surprisingly the functional $\F$ has no
minimizer among subsets of $\R^d$ of this given volume. Indeed, it is more convenient to 
spread the excess charge into little drops far away from each other.  
Such result is
contained in the following theorem:
\begin{theorem}\label{nonexistRintro}
For every $\a\in (0,d-1)$, there holds
\[
\inf_{|E|=m}\F(E)=\left(\frac{m}{\omega_d}\right)^{\frac{d-1}{d}}\P(B).
\]
\end{theorem}
\noindent
Ultimately, this comes from the fact that the perimeter is defined up to sets of Lebesgue measure zero while the
Riesz
potential energy is defined up to sets of zero capacity.
This phenomenon is further illustrated when considering the problem among sets which are contained in a fixed
bounded domain $\Omega$. In this case  we prove that the 
 isoperimetric problem and the charge minimizing
problem completely decouple.

\begin{theorem}\label{mainintro}
Let $\Omega$ be a   compact subset   of $\R^d$ with smooth boundary, and let 
$0<m<|\Omega|$. Let $E_0$ be a solution of the constrained isoperimetric problem
\begin{equation}\label{constrainedisointro}
\min \lt\{\P(E): E\subset \Omega,\ |E|=m\rt\}.
\end{equation}
Then, for $\a\in(0,d-1)$ and $Q>0$ we have
\begin{equation}\label{illintro}
\inf_{|E|=m,\, E\subset \Omega}\F(E)
=\P(E_0)+Q^2 \Q(\Omega).
\end{equation}
\end{theorem}

As a by-product of our analysis we also get that $\F$ does not have local minimizers with respect to the $L^1$ or
even Hausdorff topology: 
\begin{theorem}\label{corunstable}
For any $\a\in (0,d-1)$ and $Q>0$, the functional $\F$ does not admit 
local volume-constrained minimizers with respect to the $L^1$ or the Hausdorff topology. 
\end{theorem}
Let us stress the fact that Theorem \ref{corunstable}, asserts in particular that there is \textit{never} non-linear stability of the ball. However,
 we should also notice that the competitors that we construct and which are made of infinitely small droplets, are very singular. It would be interesting to
better  understand  the mechanism preventing the formation of such micro drops.\\

\noi  
One possible explanation is that global (or even local) $L^1$ minimizers are
not the right objects to consider.
One should instead look for stable configurations under smoother deformations. These are typically
local minimizers for a stronger topology. 
It is then reasonable to look for minimizers of $\F$ in some smaller class
of sets with some extra regularity conditions. The class  that we take into consideration, and denote by $\Kd$, is
that of sets which admit at every point of their boundary an internal and an external tangent ball of a fixed radius
$\delta$  (namely, the $\delta$-ball condition, see Definition \ref{d-ball}).  We denote by $\Kdco$
the class of connected sets of $\Kd$. The purpose of introducing such a class is to prove the stability of the ball with respect to $C^{1,1}$ perturbations.
 There are indeed mainly two (mathematical) advantages of working in $\Kd$. The first, is that it ensures density 
estimates on the sets. These estimates are usually the most basic regularity results available for minimizers of minimal surfaces
 types of problems (see \cite{giusti,GN,LuOtto,KM2}). Thanks to the constructions of Theorem \ref{nonexistRintro},
 we see that in our problem there is no hope to get such estimates without imposing them \textit{a priori}.  The second advantage is  that, at least in the Coulombic case $\a=d-2$,
for  every set $E \in \Kd$, the minimizing measure for $\Q(E)$ is a uniformly bounded measure on  $\partial E$   (see the end of Section \ref{droplets2}). We use in a crucial way 
this $L^\infty$ control on the charge density in the analysis of the stability of the ball. Our second main result is then:
\begin{theorem}\label{stabilityball}
 Let $d\ge 3$ and $\a=d-2$. Then for any $\delta>0$ and $m\ge \omega_d \delta^d$, there exists a charge $\bar Q
\lt(\frac{\delta}{m^{1/d}}\rt)>0$,  such that if 
$$\frac{Q}{m^\frac{d-1+\a}{2d}}\,\le\,\bar Q \lt(\frac{\delta}{m^{1/d}}\rt)
$$
the ball is stable for problem \eqref{pbcnocointro} under volume preserving perturbations with $C^{1,1}$ norm less than $\delta$.  
\end{theorem}
 This extends a previous 
result of M.A.~Fontelos and A.~Friedman 
\cite{fontfried}, which asserts the  stability with respect to  $C^{2,\alpha}$ perturbations. These authors also gave
 a detailed analysis of the linear stability. We remark that our proof of 
the stability of the ball is 
quite different from the one in \cite{fontfried}, and is inspired by the proofs in \cite{fuglede,KM2,cicaspa}. In particular it lies between linear and non-linear stability since it follows from the following three theorems
 asserting that for small charge $Q$, the ball is the unique minimizer in the class $\Kd$.

%
The first result is an existence theorem in the class $\Kdco$.
\begin{theorem}\label{mainexistenceintro}
For all $Q\ge 0$ problem

\begin{equation}\label{pbcintro}
\min\left\{\F(E):\,\,|E|=m,\,\, E\in\Kdco\,\right\},
\end{equation}
has a solution.
\end{theorem}
%
To avoid the strong hypothesis on the connectedness of the competitors, it is necessary to impose a bound from
above on the charge $Q$.
\begin{theorem}\label{mainexistence3intro}
There exists a constant $Q_0=Q_0(\a,d)$ such that, for every $\delta>0$, $m\ge \omega_d\delta^d$ and
$$
\frac{Q}{m^\frac{d-1+\a}{2d}}\,\le\, Q_0\, \frac{\delta^d}{ m}\,, 
$$
problem 
\begin{equation}\label{pbcnocointro}
\min\left\{\F(E):\,\,|E|=m,\,\, E\in\Kd\,\right\},
\end{equation} 
has a solution.
\end{theorem}
It is worth remarking that the main ingredient of the proof of Theorem \ref{mainexistence3intro} is the
isoperimetric inequality in quantitative form (see \cite{FMP,fimapr,cicaleseleonardi}). Finally, using delicate estimates on the Riesz potential energy $\Q(E)$ for small perturbations of the ball, we are able to prove the following stability theorem in the Coulombic case. 

\begin{theorem}\label{stability}
Let $d\ge 3$ and $\a=d-2$. Then for any $\delta>0$ and $m\ge \omega_d \delta^d$, there exists a charge $\bar Q
\lt(\frac{\delta}{m^{1/d}}\rt)>0$,  such that if 
$$\frac{Q}{m^\frac{d-1+\a}{2d}}\,\le\,\bar Q \lt(\frac{\delta}{m^{1/d}}\rt)
$$
the ball is the
unique minimizer of problem \eqref{pbcnocointro}.  
\end{theorem}
It would be interesting to understand if our stability result could be extended both to the case $\a\neq d-2$ and maybe more interestingly to a weaker class 
of perturbations such as for instance small Lipschitz ones. 

\noi Let us point out that for $\a\le d-2$, the optimal measure for the Riesz potential concentrate on the boundary of the sets whereas for $\a>d-2$ it has support on the whole set (see Lemma \ref{F=G}). Therefore, for $\a>d-2$, it makes also sense to consider the functional
\[\mathcal{G}_{\a,Q}(E)=P(E)+Q^2 \Q(\partial E)\]
for which we can prove similar results to the ones described above.\\ 


\noi Let us close this introduction by comparing our results with the analysis in \cite{KM1,KM2,cicaspa,LuOtto,bocr,julin} 
of the non-local isoperimetric problem, known as the  sharp interface Ohta-Kawasaki model,
\begin{equation}\label{ohta-kawasaki}
\min_{|E|=m} P(E)\, +\, \int_{E\times E}\frac{dx\, dy}{|x-y|^\a}\,,
\end{equation}
which is motivated by the theory of diblock copolymers and the stability of atomic nuclei. The authors show 
that there exist two (possibly equal) critical volumes $0<m_1(\alpha)\le m_2(\alpha)$
such that minimizers exist if $m\le m_1$, while there are no minimizers if $m>m_2$. Moreover, 
the minimizers are balls when $\a<d-1$ and the volume is sufficiently small. These results have been generalized to non-local perimeters in \cite{labandadei5} (see also \cite{dica_nov_ruf_val}). A crucial difference between our
model and the Ohta-Kawasaki model is that in the latter, the non-local term is Lipschitz with respect to the
measure of the symmetric difference between sets (see for instance \cite[Prop. 2.1]{cicaspa}). Hence, on small
scales,
 the perimeter dominates the non-local part of the energy. This implies in particular that minimizers enjoy the 
 same regularity properties as minimal surfaces. In our case, it is quite the contrary since on small scales,
the functional $\Q$ dominates the perimeter. As already pointed out above, this prevents \textit{a priori} the hope to get any regularity result
for stable configurations.
Let us notice that the same  type of existence/non-existence issues in variational models where the perimeter 
competes against a non-local energy has been recently addressed in other models. 
For instance, in \cite{BelGolZwick} the authors study a model related to epitaxial growth where the non-local part 
forces compactness whereas the perimeter part favor spreading.\\

The paper is organized as follows. In Section \ref{droplets2} we recall and prove
 some properties of the Riesz potentials $\Q$.    In Section \ref{droplets3}, we prove
  the non-existence of minimizers for the functional $\F$ (in particular we prove Theorems \ref{nonexistRintro}
and \ref{corunstable}). In Section \ref{droplets4}, we study this existence
issue, that is we prove Theorems \ref{mainexistenceintro} and \ref{mainexistence3intro}), before proving in Section
\ref{droplets5} the stability of the ball (Theorem \ref{stability}). 
Finally, in Section \ref{droplets6}, we extend our results to the logarithmic potential energy
\[
I_{\mathrm{log}}(E):=\inf\left\{\int_{\R^d\times \R^d}\log\left(\frac{1}{|x-y|}\right)d\mu(x)d\mu(y):\,\mu(E)=1\right\}.
\]

\section{The Riesz potential  energy}\label{droplets2}

 \noi In this section we recall some results  regarding the Riesz potential energy
(see Definition \ref{Qmu} below). Most of the material presented here comes from  \cite{landkof}. 

\noi In the following, given an open set $\Omega\subset \R^d$, we denote by $\M(\Omega)$ the set of all Borel
measures 
with support in $\Omega$. For $x\in \R^d$ and $r>0$ we denote by $B_r(x)$ the open ball of radius $r$ centered 
in $x$ and simply by $B$ the unit ball and by $\omega_d=|B|$ its Lebesgue measure. For $k\in [0,d]$, we will denote
by $\H^k$ the $k$-dimensional Hausdorff measure.

\begin{definition}\label{Qmu}
Let $d\ge 2$ and $\alpha>0$. Given $\mu,\,\nu\in\M(\R^d)$, we define the {\em interaction energy} (or {\em
potential energy}) between
$\mu$ and $\nu$ by 
\[
\Q(\mu,\nu):=\int_{\R^d\times \R^d} \frac{d\mu(x)\,d\nu(y)}{|x-y|^\a}\ \in [0,+\infty] .
\]
When $\mu=\nu$, we simply write $\Q(\mu):=\Q(\mu,\mu)$. When the measures are absolutely continuous with respect to
the Lebesgue measure, that is $\mu=f \Hd\res E$ and $\nu=g \Hd \res E$ 
for some set $E$ and functions $f$ and $g$, 
we denote  $\Q(\mu,\nu)=\I_\a^E(f,g)$ (and when $f=g$ we  denote it by $\I_\a^E(f)$). Similarly, 
when $\mu= f\Hdm\res\partial E$ and  $\nu=g \Hdm\res  \partial E$
 we write $\Q(\mu,\nu)=\I_\a^{\partial E}(f,g)$ (and when $f=g$ we denote it by $\I_\a^{\partial
E}(f)$).  
\end{definition}
The following proposition can be found in \cite[(1.4.5)]{landkof}.
\begin{proposition}
 The functional $\Q$ is lower  semicontinous for the weak* convergence of measures.
\end{proposition}
We can then define the Riesz potential energy of a set.
\begin{definition}\label{Q0}
Let $d\ge 2$ and $\alpha>0$ then for every Borel set $A$ we define the Riesz potential  energy of $A$ by
\begin{equation}\label{eqQ}
\Q(A):=\inf \lt\{\Q(\mu):\,\mu\in\mathcal M(\R^d),\,\mu(A)=1\rt\}.
\end{equation}
\end{definition}

\begin{remark}\rm
Notice that, if we change $\mu$ in $Q\mu$ for a given charge $Q>0$, then for any Borel set $A\subset\R^d$, it
holds 
\[
Q^2 \Q(A):= \inf	\lt\{\Q(\mu):\,\mu\in\M(\R^d),\,\mu(A)=Q\rt\}.
\]
Notice also that, for all $\lambda>0$, there holds
\begin{equation}\label{eqscale}
\Q(\lambda A)=\lambda^{-\a}\Q(A).
\end{equation}
\end{remark}

\begin{remark}\label{rmkcap}\rm
An important notion related to $\Q(A)$ is the so-called 
 $\alpha$-capacity \cite{liebloss,landkof,Mattila}
\[C_{d-\a}(A):=\frac{1}{\Q(A)}.\]
For $\a=d-2$ and $K$ compact, we have the following representation of the capacity \cite{liebloss}:
\begin{eqnarray*}
C_{2}(K)&=&\inf\left\{ \int_{\R^d} |\nabla f|^2 \, :\, f\in C^1_c(\R^d), \ f \ge 0,\ f \ge 1 \textrm{ on
}K\right\}.
\end{eqnarray*}
We stress however, for
the sake of completeness, that there are other notions of capacity in the literature ( see for
instance the discussion in \cite[Section $11.15$]{liebloss}).
\end{remark}

\begin{remark}\label{pallamax}\rm
  It is well known that the ball minimizes the perimeter under volume constraint. On the other hand in
\cite{betsa} it was proven that if $\a> d-2$, then the ball maximizes the Riesz Potential $\I_\a$ among compact
sets of given volume.
\end{remark}

\noi The proof of the following result is given in \cite[p. $131$ and $132$]{landkof}.
\begin{lemma}\label{Q1}
If $A$ is a compact set, 
the infimum in \eqref{eqQ} is achieved. 
\end{lemma}

%

\begin{remark}\rm
When the set $A$ is unbounded, there does not always exist an optimal measure $\mu$, i.e. the infimum in
\eqref{eqQ} is not achieved. 
Indeed, it is possible to construct  a set $E$ of finite volume with
$\Q(E)=0$. To this aim, consider $\a\in (0,d-1)$, $\gamma\in (\frac{1}{d-1},+\infty)$ and the set
$E=\{ (x,x')\in
\R\times \R^{d-1} \ : \ |x'|\le 1 \textrm{ and } |x'|\le \frac{1}{|x|^\gamma} \}$.
The set $E$ has finite volume  and taking $N$ balls of radius $r=N^{-\beta}$ inside $E$, at mutual distance 
$\ell=N^{\frac{\beta}{\gamma}-1}$, with charge $1/N$ distributed uniformly on each ball, we have
\[
\Q(E)
\le C\left( N^{\a \beta -1}+  N^{(1-\frac{\beta}{\gamma})\a}\right)
\]
for some $C>0$,
so that $\Q(E)=0$ if $\frac{1}{d-1}<\gamma<\beta<\frac{1}{\a}$. 
Similarly, if $d>2$ and $\a<d-2$, taking $\gamma>\frac{1}{d-2}$ 
one can even construct a set with finite perimeter for which the same property holds.
\end{remark}

\begin{definition}\label{potential}
Given a non-negative Radon measure $\mu$ on $\R^d$ and $\a\in(0,d)$, we define the potential function 
\[
v^\mu_\a(x):=\int_{\R^d}\frac{d\mu(y)}{|x-y|^\a}=\mu*k_\a(x)
\]
where $k_\a(x)=|x|^{-\a}$. We will sometime drop the dependence of $\mu$ and $\a$ in the definition  of $v_\a^\mu$
and we will refer to it as the  {\it potential}.
\end{definition}
\begin{definition}
 We say that two functions $u$ and $v$ are equal $\alpha$-quasi everywhere (briefly $u=v$ $\a$-q.e.) if they
coincide up to a set of $\alpha$-capacity $0$. 
\end{definition}

\noi The Euler-Lagrange equation of $\Q(A)$ reads as follows:

\begin{lemma}\label{ELmu}
Let $A$ be a compact set and let $\mu$ be a minimizer for $\Q(A)$ then $v^\mu=\Q(A)$ $\a$-q.e. on
$\spt(\mu)$, and $v^\mu\ge \Q(A)$ $\a$-q.e. on $A$. Moreover, the following equation holds 
in the distributional sense
\begin{equation}\label{eqvmu}
\left(-\Delta\right)^\frac{d-\a}{2}v^\mu = c(\a,d)\, \mu\, ,
\end{equation}
where $\left(-\Delta\right)^s $ denotes the fractional Laplacian (see \cite{Val}). In particular, 
\[
\left(-\Delta\right)^\frac{d-\a}{2}v^\mu = 0 \qquad {on\ }\R^d\setminus A\,. 
\]
\end{lemma}

\begin{proof}
The first assertions on $v^\mu$ follow from \cite[Theorem 2.6 and p. 137]{landkof}
(see also \cite{gauss} where these conditions were first derived).

Equation \eqref{eqvmu} can be directly verified by means of the Fourier Transform, namely
\[
\widehat{\left(-\Delta\right)^\frac{d-\a}{2}v^\mu} (\xi)= |\xi|^{d-\a} \widehat{\mu *k_\a}(\xi)
= c(\a,d)\,\widehat \mu(\xi)\,, 
\]
where we used the fact \cite[Equation (1.1.1)]{landkof}
\[
\widehat k_\a(\xi)=c(\a,d)\,k_{d-\a}(\xi) \qquad {\rm with}\quad
c(\a,d):= \pi^{\a-\frac d 2}\,\frac{\Gamma\left(\frac{d-\a}{2}\right)}{\Gamma\left(\frac{\a}{2}\right)}\,.
\]
\end{proof}

%

\noi We recall another important result which will be exploited in Section \ref{droplets4}.
 We refer to \cite[Theorem  1.15]{landkof} (see also \cite[Corollary 5.10]{liebloss}) for its proof.
\begin{theorem}\label{positiveenergy}
For any signed measure $\mu$ and for any $\a\in(0,d)$, there holds
\[\Q(\mu)=\int_{\R^d} \left( v_{\a/2}^{\mu}(x)\right)^2 \ dx\]
and therefore,
 \[
\Q(\mu) \ge0.
 \]
Moreover equality holds if and only if $\mu=0$.
\end{theorem}

\begin{remark}\rm
A  consequence of Theorem \ref{positiveenergy}, is that the functional $\Q(\cdot,\cdot)$ is a {\it
positive}, bilinear operator on the product space of Radon measures on $\R^d$, $\mathcal M(\R^d)\times\mathcal
M(\R^d)$. In particular it satisfies the Cauchy-Schwarz inequality
\begin{equation}\label{cs}
\Q(\mu,\nu)\le\Q(\mu)^{1/2}\Q(\nu)^{1/2}.
\end{equation}

\end{remark}
\noi The following uniqueness result can be found in \cite[page 133]{landkof}. 

\begin{lemma}\label{uniq}
For every compact set $A$ the measure minimizing $\Q(A)$ is unique.
\end{lemma}
In the next lemma, we recall some properties of the support of the optimal measures.
\begin{lemma}\label{F=G}
Let $\a\in(0,d-1)$. For every open bounded  set $E$, the  minimizer $\mu$ of $\Q(E)$ satisfies:
\begin{itemize}
\item[i)] If $\a\le d-2$ then $\spt (\mu)\subset \partial E$. In particular $\Q(E)=\Q(\partial E)$.
\item[ii)] If $\a>d-2$ then  $\spt (\mu)= \ovE$.
\end{itemize}
Moreover, when $\a\ge d-2$, $v_{\a}^\mu=\Q(E)$ on $\ovE$.
\end{lemma}

\begin{proof}
 The case $\a\le d-2$ can be found in \cite[page $162$]{landkof}. If $\a>d-2$, by \cite[Theorem
$2.6$ and page $137$]{landkof}, 
we know that $v_\a^\mu=\Q(E)$ $\a$-q.e. on $\ovE$ and $v_\a^\mu\le \Q(E)$ on $\R^d$. Moreover, outside of 
$\spt(\mu)$, $v_\a^\mu$ is smooth and $\Delta v_\a^\mu>0$. Assume that there exists $x\in E
\setminus\spt(\mu)$. Then there exists an open ball $B_r(x)\subset E\setminus\spt(\mu)$. But this is impossible 
since  this would imply $v_\a^\mu=\Q(E)$ in $B_r(x)$ and hence $\Delta v_\a^\mu=0$ in $B_r(x)$ contradicting
$\Delta v_\a^\mu>0$. The last claim of the lemma follows by the fact that $v_\a^\mu$ is, in this case, a regular
function on $E$ which is $\a$-q.e. equal to $\Q(E)$.
\end{proof}

\noi We now prove a  density result which is an adaptation of \cite[Theorem $1.11$ and Lemma $1.2$]{landkof}.

\begin{proposition}\label{density}
Let $E$ be a smooth connected closed set of $\R^d$, then for every $\a\in(0,d)$, 
\[
\Q(E)=\inf\left\{ \Q^E(\mu) \ : \ \mu=f dx, \  f\in L^{\infty}(E), \ \int_E f \ dx =1\right\}.
\] 
\end{proposition}

\begin{proof}
By Definition \ref{Q0} and Lemma \ref{Q1} the proof reduces to the approximation of $I_\a(\mu)$ for a given
measure $\mu$ supported on $E$ and such that $\mu(E)=1$.
Let $\mu$ be such that $\mu(E)=1$, $\spt(\mu)\subset E$ and $\Q(\mu)<+\infty$ then for $\eps>0$ consider the
measure $\mu_\eps\ dx$ defined as
\[
d\mu_\eps(x)=\left( \int_{B_\eps(x)\cap E} \frac{d\mu(y)}{|E\cap B_\eps(y)|}\right)\,d\H^{d}\res E.
\]
Notice that by definition $\spt(\mu_\eps)\subseteq E$. Moreover we have, by the Fubini Theorem,
\[
 \begin{aligned}
  \mu_\eps(E)&=\int_E\mu_\eps(x)\,dx=\int_E\int_E\frac{\chi_{B_\eps(x)}(y)\,d\mu(y)}{|B_\eps(y)\cap E|}\,dx\\
  &=\int_E\int_E\chi_{B_\eps(y)}(x)\,dx \frac{d\mu(y)}{|B_{\eps}(y)\cap E|}=\int_E \,d\mu(y)=1.
 \end{aligned}
\]
\noindent
Since $\|\mu_\eps\|_{L^{\infty}(E)}\le (\min_{x\in E}|B_\eps(x)\cap E|)^{-1}\le (C\eps^{d})^{-1}$, we only
have to prove that $\Q^E(\mu_\eps)\to \Q(\mu)$. By Theorem \ref{positiveenergy} we have
$$\Q^E(\mu_\eps)=\int_{\R^d} \left(v_{\a/2}^{\mu_\eps}(x)\right)^2 \ dx.
$$ 
Let us show that for all $x\in \R^d$, 
\[
v_{\a/2}^{\mu_\eps}(x)\le C v_{\a/2}^{\mu}(x) \quad \textrm{ and } \quad \lim_{\eps\to 0}
v_{\a/2}^{\mu_\eps}(x)=v_{\a/2}^{\mu}(x)
\]
 from which we can conclude by means of the Dominated Convergence Theorem. Denoting by $\chi_A$ the
characteristic function of the set $A$, we have, for any $x\in
\R^d$,
\begin{equation}\label{rev1}
\begin{aligned}
 v_{\a/2}^{\mu_\eps}(x)&=\int_{E} \int_{E} \frac{1}{|B_\eps(y)\cap E|} \chi_{B_\eps(y)}(z)
\frac{d\mu(z)}{|x-y|^{\a/2}} \, dy\\
&= \int_{E} \left(\int_{B_\eps(z)\cap E}  \frac{1}{|B_\eps(y)\cap E|}
\frac{|x-z|^{\a/2}}{|x-y|^{\a/2}} \ dy\right) \frac{d\mu(z)}{|x-z|^{\a/2}} \\
&\le \int_{E} \left( \frac{C}{\eps^d} \int_{B_\eps(z)}  \frac{|x-z|^{\a/2}}{|x-y|^{\a/2}} \
dy\right) \frac{d\mu(z)}{|x-z|^{\a/2}}. 
\end{aligned}
\end{equation}
Moreover it is possible to prove that the function 
\begin{equation}\label{rev2}
 (x,z,\varepsilon)\mapsto \eps^{-d}
\int_{B_\eps(z)}  \frac{|x-z|^{\frac{\a}{2}}}{|x-y|^{\frac{\a}{2}}} \ dy
\end{equation}
is uniformly bounded in $(x,z,\eps)$
(see  \cite[Theorem  $1.11$]{landkof})
so
that
$v_{\a/2}^{\mu_\eps}(x)\le C v_{\a/2}^{\mu}(x)$ for a suitable constant $C>0$. Consider now a
point $x\in\R^d$ such that $v^\mu_{\a/2}(x)<+\infty$. Then for every $\delta>0$ 
there is a ball $B_\eta(x)$ such that $v_{\a/2}^{\mu'}<\delta$ where $\mu'=\mu\res B_\eta(x)$. 
By the previous computations, we know that $v_{\a/2}^{(\mu')_\eps}(x)\le C \delta$. Moreover,  $\lim_{\eps\to
0} v_{\a/2}^{(\mu-\mu')_\eps}(x)=v_{\a/2}^{\mu-\mu'}(x)$. 
Indeed, denoting for simplicity $\nu:=\mu-\mu'$, we have
that
\[
\begin{aligned}
 v_{\a/2}^{\nu_\eps}(x)&=\int_E \frac{d\nu_\eps(y)}{|x-y|^{\a/2}}=\int_E\int_E
\frac{\chi_{B_\eps(y)}(z)\,d\nu(z)}{|B_\eps(z)\cap E|}\frac{dy}{|x-y|^{\a/2}}\\
&=\int_E\int_E
\frac{\chi_{B_\eps(z)}(y)\,d\nu(z)}{|B_\eps(z)\cap E|}\frac{dy}{|x-y|^{\a/2}}\\
&= \int_E\left(\frac{1}{|B_\eps(z)\cap E|}\int_{E\cap B_\eps(z)}\frac{dy}{|x-y|^{\a/2}}\right)
d\nu(z).
\end{aligned}
\]
>From this the claim follows since the last quantity inside the parentheses uniformly converges to the function $|x-z|^{-\a/2}$ on every compact
set which does not contain $x$, and since $\spt(\nu)=\spt(\mu-\mu')\subset B(x,\eta)^c$.

\noindent
Furthermore, we have that
$v_{\a/2}^{(\mu-\mu')_\eps}=v_{\a/2}^{\mu_\eps}-v_{\a/2}^{\mu'_\eps}$. Thus we get
\begin{equation}\label{ref3}
 \begin{aligned}
 v^{\mu}_{\a/2}(x)&=v^{\mu'}_{\a/2}(x)+v^{\mu-\mu'}_{\a/2}(x)\le \delta+ \lim_{\eps\to 0}
v_{\a/2}^{(\mu-\mu')_\eps}(x)\\
&\le (1+C)\delta +\varliminf_{\eps\to 0} v_{\a/2}^{\mu_\eps}(x)\le (1+C)\delta +\varlimsup_{\eps\to 0}
v_{\a/2}^{\mu_\eps}(x)\\
&\le (1+C)\delta+\varlimsup_{\eps\to 0} v_{\a/2}^{\mu'_\eps}(x)+\varlimsup_{\eps\to 0}
v_{\a/2}^{(\mu-\mu')_\eps}(x)\\
&\le 2(1+C)\delta +v_{\a/2}^\mu(x)
\end{aligned}
\end{equation}
so that letting $\delta\to 0$ we get that $\lim_{\eps\to 0}v_{\a/2}^{\mu_\eps}(x)=v_{\a/2}^{\mu}(x)$ as claimed.
\end{proof}

\noi For the unit ball, since the problem is invariant by rotations, it is not hard to compute the exact minimizer
of $\Q(B)$ or $\Q(\partial B)$, see \cite[Chapter II.13]{landkof}.
\begin{lemma}\label{ball}

The uniform measure on the sphere $\partial B$ 
\[
d\mathcal{U}_B=\frac{1}{\P(B)}d\Hdm\res{{\partial B}}
\]
is the unique optimizer for $\Q(\partial B)$. For $d>\a> d-2$, the measure 
\[
d\mathcal{\tilde U}_B=\frac{C_\a}{(1-|x|^2)^{\frac \a 2}}d\Hd\res{{ B}}
\]
is the unique optimizer for $\Q(B)$ (where $C_\a$ is a suitable renormalization constant).
\end{lemma}


\begin{definition}\label{d-ball}
Given $\delta>0$, we say that $E$ satisfies the {\em internal $\delta$-ball condition} if for any $x\in\pa E$
there is a ball of radius $\delta$ contained in $E$
and tangent to $\pa E$ in $x$. Analogously, $E$ satisfies the {\em external $\delta$-ball condition} if for any
$x\in\pa E$, there is a
ball of radius $\delta$ contained in $E^c$. Finally, if $E$ satisfies both the internal and the external
$\delta$-ball condition we shall say that it satisfies the {\em $\delta$-ball condition}.
\end{definition}
 We remark that the sets which satisfies the $\delta$-ball condition have  $C^{1,1}$ boundary with principal
curvatures
bounded from above by $1/\delta$, see \cite{DelZol}.
\noi We denote by $\Kd$ the class of all the closed  sets which satisfy the $\delta$-ball condition and by
$\Kdco$ the subset of $\Kd$ 
composed of connected sets. 


\begin{remark}\rm
An equivalent formulation of Definition \ref{d-ball} is requiring
that $d_E\in C^{1,1}(\{|d_E|<\delta\})$, where
\[
d_E(x)=\left\{
\begin{array}{ll}
{\rm dist}(x,\partial E) & if\ x\not\in E
\\
- {\rm dist}(x,\partial E) & if\ x\in E
\end{array}\right.
\]
is the signed distance function from $\partial E$. See for instance \cite{DelZol}.
\end{remark}

\begin{lemma}\label{bound}
 Let $\delta>0$, then every set $E\in \Kdco $ with $|E|=m$ satisfies
\[
 \mathrm{diam} (E)\le  \sqrt{d} \, 2^{d+2} \,\frac{m}{\omega_d}\, \delta^{1-d}  .
\]
\end{lemma}
\begin{proof}
Consider the tiling of $\R^d$ given by $[0,2\delta)^d+2\delta\Z^d$ and for $k\in \Z^d$ let
$C_k=[0,2\delta)^d+2\delta k$.
For every $k\in \Z^d$ such that $C_k\cap E\neq \emptyset$, let $B_\delta(x_k)$ be a ball of radius $\delta$ such
that $B_\delta(x_k)\subset E$ and $B_\delta(x_k)\cap C_k \neq \emptyset$. 
The existence of such a ball is guaranteed by the $\delta$-ball condition. Any such ball can intersect at most
$2^d$ cubes $C_j$ so that
\[
\sharp \{k\in \Z^d \, : \, E\cap C_k \neq \emptyset\}=\frac{1}{|B_\delta|} \sum_{k: C_k\cap E\neq \emptyset}
|B_\delta(x_k)|\le \frac{2^d}{|B_\delta|} |E|,
\]
where $\sharp A$ is the cardinality of the set $A$.
 The fact that $E$ is connected implies that, up to translation, $E\subset [0,4\delta\frac{2^d}{|B_\delta|} m]^d$.
Thus we can conclude that 
\[
 \mathrm{diam} (E)\le \mathrm{diam}\left( \left[0,4\delta\frac{2^d}{|B_\delta|} m\right]^d\right)=\sqrt{d} \,
2^{d+2}\,\frac{m}{\omega_d}\, \delta^{1-d}.
\]
\end{proof}

\begin{remark}\rm
 As already pointed out in the introduction, in some sense the $\delta$-ball condition is the analog of the famous density estimates for problem in which
the
perimeter term is dominant see \cite{giusti}. Since, in the problems we are going to consider, both the
perimeter and the Riesz potential energy are 
of the same order, there is {\it a priori} no hope to get such density estimates from the minimality. It is a
classical feature that for connected sets, these density estimates provide a bound on the diameter \cite{GN}.
\end{remark}

\begin{proposition}\label{regolarita}
Let $d\ge 3$, $\a=d-2$, $\delta>0$ and $E\subset\R^d$ be a compact set which satisfies the $\delta$-ball
condition. 
Then the optimal measure $\mu$ for $\Q(E)=\Q(\partial E)$ can be written as  $\mu= f \Hdm\res\partial E$ with
$\LinfE{f}\le \Q(E)(d-2)\delta^{-1}$.
\end{proposition}

\begin{proof}
By Lemma \ref{F=G} we know that the optimizer $\mu$ is concentrated on
$\partial E$.
Denote by $v=v_{d-2}^\mu$ the potential related to $\mu$ on $E$. By Lemma \ref{F=G}, we know that $v=\Q(E)$ on
$E$,
and that $-\Delta v=\mu$. 
By classical elliptic regularity (see for instance \cite[Cor. 8.36]{gilTru}), $v$ is regular in $\R^d\backslash E$,
and $C^{1,\beta}$ up to the boundary of $E$.
Consider now a point $x\in\partial E$ and let $y\in E$ such that the ball $B_\delta(y)$ is contained in $E$ and is
tangent to $\partial E$ in $x$. The existence of such a $y$ is guaranteed by the $\delta$-ball condition
satisfied by $E$. Let $u$ be a solution of 
\[
 \Delta u=0\quad \textrm{ in} \,\, B_\delta^c(y);\qquad u=v(x)=\Q(E)\quad \textrm{on }\,\,\partial B_\delta(y). 
\]
Notice that $u(z)=\frac{\Q(E)\delta^{d-2}}{|z-y|^{d-2}}$ out of $B_\delta(y)$. By the maximum principle for
harmonic functions, $u\le \Q(E)$ on $\partial E$. Thus, again by the maximum principle,
applied to $u-v$, we get that $v\ge u$ on $\R^d\setminus E$. Since $u(x)=v(x)$, 
\begin{equation}\label{barrier1}
 |\nabla v(x)|\le|\nabla u(x)|=\Q(E)(d-2)\delta^{-1}.
\end{equation}

\noi Let us prove that $\mu=|\nabla v| \Hdm\res \partial E$. For this, let $x\in \partial E$ and $r>0$ and
consider
a test  function
 $\vphi\in C^{\infty}_c(\R^d)$. Then we have
\begin{equation}\label{barrier2}
 \begin{aligned}
 \int_{\partial E} \vphi d\mu&=-\int_{\R^d} \vphi \Delta v =\int_{\R^d} \langle\nabla \vphi, \nabla v\rangle \
dy\\
&=\int_{E^c} \langle\nabla \vphi, \nabla v\rangle \ dy=\int_{\partial E} \vphi \langle\nabla v, \nu^E
\rangle d \Hdm
\end{aligned}
\end{equation}
\noi
where $\nu^E$ is the external normal to $E$. Since $v$ is constant on $\partial E$, its  tangential derivative
is 
zero. Thus, since $v<\Q(E)$ on $\R^d\backslash \ovE$ we have that  $\langle\nabla v, \nu^E\rangle\ge 0$.
Therefore, $\langle\nabla v, \nu^E\rangle=|\nabla v|$ on $\partial E$. Hence, by \eqref{barrier2} we conclude that
for every test function $\vphi$,
\[\int_{\partial E} \vphi d\mu= \int_{\partial E} \vphi |\nabla v| d \Hdm,\]
which is equivalent to the claim $\mu=|\nabla v| \Hdm\res \partial E$.

\end{proof}

\section{Non-existence of minimizers}\label{droplets3}
 
\begin{definition}
Let $d\ge2$ and $\alpha>0$. For every $Q>0$ and every open  set $E\subset\R^d$ we define the functionals,
\begin{equation}\label{F}
\F(E):=\P(E)+Q^2 \Q(E),
\end{equation}
and
\begin{equation}\label{G}
 \G(E):=\P(E)+Q^2 \Q(\partial E).
\end{equation}
\end{definition}

\noi Notice that by Lemma \ref{F=G}, for $\a\in(0,d-2]$ the functionals $\F$ and $\G$ coincide.
Notice also that $\F(E)\equiv +\infty$ if $\a\ge d$, and $\G(E)\equiv +\infty$ if $\a\ge d-1$.

\noi In this section we consider a closed, connected, regular set $\O\subset \R^d$ (not necessarily bounded)
of
measure $|\Om|>m$  and address the following problems:
\begin{equation}\label{pb}
\inf_{|E|=m,\, E\subset\Omega}\F(E),
\end{equation}
and 
\begin{equation}\label{pbb}
\inf_{|E|=m,\, E\subset\Omega}\G(E),
\end{equation}
where the (implicit) parameter $\a$ belongs to $(0,d)$. 


\begin{theorem}\label{nonexistR}
For every $\a\in (0,d-1)$, there holds
\[
\inf_{|E|=m}\F(E)=\inf_{|E|=m}\G(E)=\min_{|E|=m} \P(E)=\left(\frac{m}{\omega_d}\right)^{\frac{d-1}{d}}\P(B).
\]
In particular, problems \eqref{pb} and \eqref{pbb} do not admit minimizers when $\Omega=\R^d$.
\end{theorem}

\begin{proof}
Let $N\in \N$ and consider a number $\beta$ which will be fixed later on. Consider $N$ balls of radius
$r_N=N^{-\beta}$
which we can consider mutually infinitely far away (since sending them away leaves unchanged the perimeter
and decrease the potential interaction energy), and put on each of these balls a charge $\frac{1}{N}$. Let $V_N= N r_N^d \omega_d$ 
be their total volume and consider the set $E$  given by the union of these balls with a (non-charged) 
ball of volume $m-V_N$. If we choose $\beta\in(1/(d-1),1/\a)$, then we get
\begin{equation}\label{condbeta} 
\lim_{N\to +\infty} N r_N^{d-1}= 0 \qquad \textrm{and} \qquad \lim_{N\to +\infty} \frac{1}{N} \frac{1}{r_N^\a}=0.
\end{equation}
which implies that $V_N\to 0$ and 
\[
\left(\frac{m}{\omega_d}\right)^{\frac{d-1}{d}}\P(B)\le \P(E)+Q^2 \Q(E)\le
\left(\frac{m-V_N}{\omega_d}\right)^{\frac{d-1}{d}}\P(B)+C\lt( Nr_N^{d-1}+ 
\frac{Q^2}{N} \frac{1}{r_N^\a}\rt).
\]
Since the  right-hand
side
converges to $ \left(\frac{m}{\omega_d}\right)^{\frac{d-1}{d}}\P(B)$, as $N$ tends to $+\infty$, the claim follows.
\end{proof}

The following result follows directly from the construction made in the previous theorem.

\begin{corollary}
Let $\alpha \in (0,d-1)$ and $m>0$.  For every $0<\delta<(m/\omega_d)^\frac 1 d$ there exists a charge $Q_\delta=Q_\delta(\a,m)$
such that $Q_\delta\to 0$ as $\delta\to 0$, and the ball of volume $m$ is not the minimizer of $\F$, 
among sets in $\Kd$ with volume $m$ and charge $Q>Q_\delta$. 
\end{corollary}

We now consider the case of bounded  $\O$ where the situation is more involved. 

\begin{theorem}\label{main}
Let $\Omega$ be a   compact subset   of $\R^d$ with smooth boundary, and let 
$0<m<|\Omega|$. Let $E_0$ be a solution of the constrained isoperimetric problem
\begin{equation}\label{constrainediso}
\min \lt\{\P(E): E\subset \Omega,\ |E|=m\rt\}.
\end{equation}
Then, for $\a\in(0,d-1)$ and $Q>0$ we have
\begin{equation}\label{ill}
\inf_{|E|=m,\, E\subset \Omega}\F(E)=\inf_{|E|=m,\, E\subset \Omega}\G(E)
=\P(E_0)+Q^2 \Q(\Omega).
\end{equation}

\end{theorem}

\begin{proof} 

We divide the proof into three steps. \\

\noi{\it Step 1.}  For $\eps>0$ and $f\in L^\infty(\Omega )$, with $f\ge 0$ and $ \ds\int_\Omega  f dx=1$, 
we shall construct a measure $\tilde\mu_\eps$ with $\spt (\tilde\mu_\eps)\subset  \Omega$, 
$\tilde\mu_\eps(\Omega)=1$, satisfying  
\begin{equation}\label{eqsupp}
\P(\spt(\tilde\mu_\eps))\le \eps 
\end{equation}
and 
\begin{equation}\label{eqff}
\Q(\tilde\mu_\eps)\le \I_\a^ \Omega(f)+\eps.
\end{equation}
Let $\delta>\lambda>0$ be small parameters to be fixed later and consider the tiling of the space given by 
$[0,\lambda)^d+\lambda\Z^d$. For every $k\in \Z^d$ such that $(\lambda k+[0,\lambda)^d)\cap \Omega\ne
\emptyset$,
we let $C_k=\lambda k+[0,\lambda)^d$ and denote by $x_k$ be the centre of $C_k$.
Notice that the number $N$ of such squares $C_k$ is bounded by $C(\Omega )\lambda^{-d}$. 
Letting $\ds f_k:=\int_{C_k} f \ dx$, it holds
\begin{equation}\label{I0}
\begin{aligned}
\sum_{|x_k-x_j|\ge2\delta}  \frac{f_k f_j}{|x_k-x_j|^\a}&=\sum_{|x_k-x_j|\ge2\delta} \int_{C_k\times C_j}
\frac{f(x) f(y)}{|x-y|^\a}\, \frac{|x-y|^\a}{|x_k-x_j|^\a} dx\ dy
\\
&\le \sum_{|x_k-x_j|\ge2\delta} \int_{C_k\times C_j}
\frac{f(x) f(y)}{|x-y|^\a}\, \frac{\left(|x_k-x_j|+2\lambda\right)^\a}{|x_k-x_j|^\a} dx\ dy
\\
&\le\sum_{|x_k-x_j|\ge2\delta} \int_{C_k\times C_j}
\frac{f(x) f(y)}{|x-y|^\a}\, \left( 1+C(\alpha)\frac\lambda\delta\right) dx\ dy
\\
\end{aligned}\end{equation}
where 
we used the fact that
\[
\sum_{|x_k-x_j|\ge2\delta}\int_{C_k\times C_j}
\frac{f(x) f(y)}{|x-y|^\a}\,dxdy\le\int_{\Omega \times \Omega }
\frac{f(x) f(y)}{|x-y|^\a}\,dxdy=\I_\a^\Omega (f)<\infty.
\]
\noi
Let now $r=(\lambda/2)^\beta$, with $\beta>1$. If ${\rm dist}(x_k,\R^d\setminus \Omega)\le r$, 
we replace the point $x_k$ with a point $\tilde x_k\in C_{j(k)}$, with $|\tilde x_k-x_{j(k)}|\ge \lambda/4$,
where $C_{j(k)}\subset \Omega$ is a cube adjacent to $C_k$.
For simplicity of notation, we still denote $\tilde x_k$ by $x_k$.
\noi
We consider $N$ balls of radius $r$ centered at the points $x_k$, and we set
\[
\tilde\mu_\eps:= \sum_{k} \frac{f_k}{\H^{d-1}(\partial B_r)} \chi_{\partial B_r(x_k)}.
\] 
Notice that such measures are suitable competitor in the definition of both the minima appearing in the definition
of $\mathcal{F}_{\a,Q}$ and $\mathcal{G}_{\a,Q}$.
By construction it holds ${\rm spt}(\tilde\mu_\eps)\subset \Omega$ and
$\tilde\mu_\eps(\Omega)= \ds\int_\Omega f dx=1$.
We have
\begin{align*}
\Q(\tilde\mu_\eps)&=\sum_{j,k} \frac{f_k f_j}{\H^{d-1}(B_r)^2} \int_{\partial B_r(x_j)\times \partial B_r(x_k)}
\frac{d\H^{d-1}(x)\, d\H^{d-1}(y)}{|x-y|^\a}
\\
&= \sum_k\frac{f_k^2}{\H^{d-1}(B_r)^2} \int_{\partial B_r(x_k)\times \partial B_r(x_k)} \frac{d\H^{d-1}(x)\,
d\H^{d-1}(y)}{|x-y|^\a}
\\
&+\sum_{|x_j-x_k|<2\delta,\,k\ne j} \frac{f_k f_j}{\H^{d-1}(B_r)^2} \int_{\partial B_r(x_j)\times \partial
B_r(x_k)} \frac{d\H^{d-1}(x)\, d\H^{d-1}(y)}{|x-y|^\a}
\\
&+\sum_{|x_j-x_k|\ge 2\delta} \frac{f_k f_j}{\H^{d-1}(B_r)^2} \int_{\partial B_r(x_j)\times \partial B_r(x_k)}
\frac{d\H^{d-1}(x)\, d\H^{d-1}(y)}{|x-y|^\a}\\
&=I_1+I_2+I_3.
\end{align*}
Moreover we have that
\begin{equation}\label{I1}
I_1 	\le C N \|f\|^2_{L^\infty(\Omega)} |C_k|^2\frac{1}{r^\a} \le C
\|f\|^2_{L^\infty(\Omega)}\lambda^{d-\a\beta},
\end{equation}
and 
\begin{equation}\label{I2}
I_2 \le C \delta^d N^2 \|f\|^2_{L^\infty(\Omega)} |C_k|^2\frac{1}{\lambda^\a}\le
C\|f\|^2_{L^\infty(\Omega)}\frac{\delta^d}{\lambda^\a}.
\end{equation}
Eventually, from \eqref{I0} it follows
\begin{equation}\label{I3}
\begin{aligned}
I_3 &= \sum_{|x_j-x_k|\ge 2\delta} 
\frac{f_k f_j}{|x_k-x_j|^\a}\frac{1}{\H^{d-1}(B_r)^2} \int_{\partial B_r(x_j)\times \partial B_r(x_k)}
\frac{|x_k-x_j|^\a}{|x-y|^\a} \
d\H^{d-1}(x)\,d\H^{d-1}(y)\\
&\le \sum_{|x_k-x_j|\ge2\delta} 
\frac{f_k f_j}{|x_k-x_j|^\a}\left(1+C(\alpha)\frac{r}{\delta}\right)\\
&\le \I_\a^\Omega (f)\left( 1+C(\alpha)\frac\lambda\delta\right) \left( 1+C(\alpha)\frac{r}{\delta}\right)\\
&\le  \I_\a^\Omega (f) + C(\a) \I_\a^\Omega (f)\frac{\lambda}{\delta}.
\end{aligned}\end{equation}
Letting $\lambda=\delta^\gamma$, from \eqref{I1}, \eqref{I2}, \eqref{I3} we then get
\[
\Q(\tilde\mu_\eps)=I_1+I_2+I_3\le  \I_\a^\Omega (f) + C(\a) \I_\a^\Omega (f)\delta^{\gamma-1}+C
\|f\|^2_{L^\infty(\Omega)}\left(\delta^{\gamma(d-\alpha\beta)}+\delta^{d-\alpha\gamma}\right).
\]
Choosing $1<\beta< d/\alpha$ and $1<\gamma< d/\alpha$, for $\delta$ small enough we obtain \eqref{eqff}.

We now show that \eqref{eqsupp} also holds. To this aim, we notice that
\begin{equation}\label{estro}
\H^{d-1}(\spt(\tilde\mu_\eps))\le CNr^{d-1}=CN\lambda^{\beta(d-1)}=C\lambda^{\beta(d-1)-d}
\end{equation}
so that, for $\lambda$ small enough, \eqref{eqsupp} follows from \eqref{estro} 
by letting $d/\alpha>\beta>d/(d-1)$, choice which is allowed  since $\alpha<d-1$.\\

\noi{\it Step 2.} Let now $E_0$ be a solution of the constrained isoperimetric problem \eqref{constrainediso}, 
and 
let $$E_\eps:= \left(E_0\cup \bigcup_{k} B_r(x_k)\right)\backslash B_\eta,
\qquad \quad \mu_\eps:=\frac{\tilde\mu_\eps\res{E_\eps}}{1-\tilde\mu_\eps(B_\eta)}\,,
$$ 
where $B_\eta\subset E_0$ is a ball such that $|E_\eps|=m$. Notice that ${\rm spt}(\mu_\eps)\subset E_\eps$ and 
$\mu_\eps(E_\eps)=1$.
\noi
Since
\[
 |B_\eta|=\left|E_0\cup\bigcup_{k} B_r(x_k)\right|-|E_\eps|\le\left|\bigcup_{k} B_r(x_k)\right|,
\]
by \eqref{estro} we have
$$|B_\eta|^{\frac{d-1}{d}}\le \left|\bigcup_{k} B_r(x_k)\right|^{\frac{d-1}{d}}\le C \P\left(\bigcup_{k}
B_r(x_k)\right)\le C\lambda^{\beta(d-1)-d},$$
so that $\eta\le C\lambda^{\beta-\frac{d}{d-1}}$. In particular, recalling \eqref{eqff},
for $\lambda$ sufficiently small the measure $\mu_\eps$ satisfies
\begin{equation}\label{eqffbis}
\Q(\mu_\eps) \le \Q(\tilde\mu_\eps) +\eps
\le \I_\a^ \Omega(f)+2\eps.
\end{equation} 
{}From \eqref{eqffbis} we then get
\begin{equation}\label{eqfuko}
\varlimsup_{\eps \to 0} \P(E_\eps)+Q^2\Q(\mu_\eps)=\P(E_0)+Q^2 \Q^\Omega(f).
\end{equation}

\smallskip

\noi{\it Step 3.}
By Proposition \ref{density} we can find a function $f\in L^\infty(\Omega)$ such that $\ds \int_\Omega f  dx=1$ and 
$\I_\a^\Omega(f)\le \Q(\Omega)+\eps$. Thus \eqref{ill} follows by \eqref{eqfuko} and a diagonal argument.

\end{proof}

Thanks to Theorem \ref{main} we are able to prove 
\begin{theorem}\label{corunstablemain}
For any $\a\in (0,d-1)$ and $Q>0$, the functional $\F$ does not admit 
local volume-constrained minimizers with respect to the $L^1$ or the Hausdorff topology. 
\end{theorem}

\begin{proof}
Let $K$ be a compact set, and let $\Om_\eps$, for $\eps>0$, be a family of open sets with smooth boundary,
such that $K\subset\Omega_\eps$ for any $\eps>0$, 
and $\Omega_\eps\to K$ as $\eps\to 0$ in the Hausdorff topology (in particular $|\Omega_\eps\setminus K|\to 0$ as $\eps\to 0$).
By Theorem \ref{main}, it is enough to show that $\Q(\Omega_\eps)<\Q(K)$ for any $\eps>0$ (with strict inequality),
which follows directly from Lemma \ref{F=G}.
\end{proof}

\begin{remark}\rm
Notice that when $\a\in(d-2,d-1)$, Problem \eqref{pbb} relaxes to its ``natural'' domain, in the sense that the
infimum is $P(E_0)+Q^2\Q(\Omega)$ and not $P(E_0)+Q^2\Q(\partial \Omega)$ as one might expect.
\end{remark}

\begin{remark}\rm
Notice also that as soon as $\Omega$ contains a ball of volume $m$ then the solution of the isoperimetric problem \eqref{constrainediso} is a ball.
\end{remark}

\begin{remark}\rm
In the statement  of Theorem \ref{main} it is possible to replace $P(E)$ by the relative perimeter $P(E;\Omega)$ (see for instance \cite{amfupa}) almost without changing the proof.
 In other words, under the hypotheses of Theorem \ref{main} we have that
\begin{equation}
\inf_{|E|=m,\, E\subset \Omega}P(E;\Omega)+Q^2\Q(E)=\inf_{|E|=m,\, E\subset \Omega}P(E;\Omega)+Q^2\Q(\partial E)
=\P(E_\Omega;\Omega)+Q^2 \Q(\Omega),
\end{equation}
being $E_\O$ a solution of the relative isoperimetric problem
\[
 \min_{E\subset\O,|E|=m}P(E;\O).
\]

\end{remark}

\begin{remark}\rm
An interpretation of Theorem \ref{main} is that Problem \eqref{ill} decouples into the isoperimetric problem
\eqref{constrainediso}
and the {\it charge-minimizing} problem \eqref{eqQ}, which are minimized separately. 
This is essentially due to the fact that the perimeter is defined up to a set of zero Lebesgue measure, while
the Riesz potential energy is defined up to a set of zero capacity \cite[Chapter $2$]{landkof}.\\ 
A consequence of this is that the minimum problem  
\[
\min\left\{\F(E):\,|E|=m,\, E\subset A\right\}
\]
 has in general no solution.
\end{remark}
\begin{remark}\rm
 For $\a \in[d-1,d)$, it  seems difficult to construct a sequence of open sets with vanishing perimeter but of
positive capacity. This is due to the fact that sets of positive $\a$-capacity have Hausdorff measure at least $\a$
(see \cite{Mattila}). 
As a consequence, the infimum of \eqref{ill} should be strictly larger than $P(E_0)$. In order to study the
question 
of existence or non-existence of minimizers, one would need to extend the definition of $\F$ to  sets which are not
open. There are mainly two possibilities to do it. The first is to let for every Borel set $E$
\[\F(E):=P(E)+Q^2 \Q(E)\]
where now $P(E)$ denotes the total variation of $\chi_E$ (see \cite{amfupa}). It is easy to see that the problem
is 
still ill posed in this class. Indeed, for every set $E$, it is possible to consider a set $F$ of positive
$\a$-capacity but of Lebesgue measure zero so that  $\F(E\cup F)<\F(E)$.
  The second possibility would be to consider the relaxation  of the functional $\F$ defined on open sets for a
  suitable topology. Because of  the previous discussion, we see that the  $L^1$ topology, for which the perimeter
has good compactness and lower semicontinuity properties, is not the right one.
The Hausdorff topology might be more adapted to this situation. Unfortunately, the resulting functional seems 
hard to identify.
\end{remark}

\begin{remark}\rm
When considering a bounded domain $A$ it is also interesting to study the Riesz potential associated to the
Green kernel 
$G_A$, with Dirichlet or Neumann boundary conditions. Since 
\[G_A(x,y)=k_{d-2}(|x-y|)+ h(x,y)\]
with $h$ harmonic in $A$ (see \cite[Chapter $1.3$]{landkof}, \cite{cicaspa}), Theorem \ref{main} can be easily
extended to that case. 
\end{remark}

\section{Existence of minimizers under some regularity conditions}\label{droplets4}

\noi In the previous section we have seen that we cannot hope to get existence for Problem \eqref{pb} without some
further assumptions on the class of minimization. In this section we investigate the existence of minimizers in the
classes $\Kd$ and $\Kdco$, defined in Definition \ref{d-ball}. More precisely, we consider the following problems:

\begin{equation}\label{pbc}
\min\left\{\F(E):\,\,|E|=m,\,\, E\in\Kdco\,\right\},
\end{equation}
\begin{equation}\label{pbcb}
\min\left\{\G(E):\,\,|E|=m,\,\, E\in\Kdco\,\right\},
\end{equation}
\begin{equation}\label{pbcnoco}
\min\left\{\F(E):\,\,|E|=m,\,\, E\in\Kd\,\right\},
\end{equation} 
\begin{equation}\label{pbcbnoco}
\min\left\{\G(E):\,\,|E|=m,\,\, E\in\Kd\,\right\}.
\end{equation}

\noi Notice that, up to rescaling, we can always assume that $|E|=\omega_d$. Indeed, 
if we let $\tilde E:= \left(\frac{\omega_d}{m}\right)^{1/d} E$, so that $|\tilde E|= \omega_d$,
from \eqref{eqscale} we get
\begin{align}\label{scaleF}
\F(E)&=\F\left(\left(\frac{m}{\omega_d}\right)^{1/d}\tilde E\right)
= \left(\frac{m}{\omega_d}\right)^\frac{d-1}{d}
{\mathcal F}_{\alpha, \left(\frac{\omega_d}{m}\right)^\frac{d-1+\alpha}{2d}Q}(\tilde E)
\\ \label{scaleG}
\G(E)&=\G\left(\left(\frac{m}{\omega_d}\right)^{1/d}\tilde E\right)
= \left(\frac{m}{\omega_d}\right)^\frac{d-1}{d}
{\mathcal G}_{\alpha, \left(\frac{\omega_d}{m}\right)^\frac{d-1+\alpha}{2d}Q}(\tilde E).
\end{align}

\begin{definition}
For any set $E$ with $|E|=\omega_d$, we let $\dP(E):=\P(E)-\P(B)\ge0$ be the {\em isoperimetric deficit} of $E$.
\end{definition}

\begin{theorem}\label{mainexistence}
For all $Q\ge 0$ problem \eqref{pbc} and \eqref{pbcb} have a solution.
\end{theorem}
\begin{proof}
Let us focus on \eqref{pbc} since the proof of the existence for \eqref{pbcb} is very similar. Let $E_n\in \Kdco$ 
be a minimizing sequence, with $|E_n|=\omega_d$. And let $\mu_n$ be the corresponding optimal measures for
$\Q(E_n)$.
Since $P(E_n)+Q^2\Q(E_n)\le P(B)+Q^2\Q(B)$, we have that
\[
  \dP(E_n)\le Q^2 \Q( B),
\]
therefore $\P(E_n)$ is uniformly bounded. By Lemma \ref{bound}, the sets $E_n$ are also uniformly bounded so that
by the compactness criterion for functions of bounded variation (see for instance \cite{amfupa}), there exists a
subsequence converging in $L^1$ to some set $E$
with $|E|=m$. Similarly, up to
subsequence, $\mu_n$ is weakly* converging to some probability measure $\mu$.

\noi Let us prove that $E_n$ converges to $E$ also in the Kuratowski convergence, or
equivalently, in the Hausdorff metric (see for instance \cite{ambrosiotilli}). Namely we have to check  the
following two conditions:
\[
\begin{aligned}
 &(i)\, x_n\to x,\quad x_n\in E_n\ \Rightarrow\ x\in E;\\
 &(ii)\, x\in E\ \Rightarrow \ \exists x_n\in E_n\, \text{such that}\ x_n\to x.
\end{aligned}
\]
The second condition is an easy consequence of the $L^1$-convergence.
To prove the first one, we notice that by the internal $\delta$-ball condition, up to choose a radius $r$ small
enough
there exists a constant $c=c(d,\delta)>0$ such that $|B(x_n,r)\cap E_n|\ge c r^d$ which implies, together with the
$L^1$-convergence, that a limit point $x$ must be in $\ov E$.  Similarly one can also prove the Hausdorff
convergence of $\partial E_n$ to
$\partial E$. 
Since the family $\Kdco$ is stable under Hausdorff convergence, we get $E\in \Kdco$.

\noi Recalling that $\P$ is lower semicontinuous under $L^1$ convergence, and  $\Q(\mu)$ is lower semicontinuous
under
weak*-convergence (for the kernel is a positive function, and thus $\Q(\cdot)$ is the supremum of continuous
functional over $\mathcal M$), we have
\[
\varliminf_{n\to +\infty} \P(E_n)+Q^2 \Q(\mu_n)\ge \P(E)+Q^2 \Q(\mu).
\]
By the Hausdorff convergence of $E_n$, there also holds $\spt(\mu)\subset E$, which concludes the proof.
\end{proof}

\noi Thanks to the quantitative isoperimetric inequality \cite{FMP}, we can also prove existence for small charges
of minimizers even without assuming {\it a priori} the connectedness. This  is reminiscent of
\cite{KM1,KM2,cicaspa}. 

\begin{theorem}\label{mainexistence3}
There exists a constant $Q_0=Q_0(\a,d)$ such that, for every $\delta>0$, $m\ge \omega_d\delta^d$ and
$$
\frac{Q}{m^\frac{d-1+\a}{2d}}\,\le\, Q_0\, \frac{\delta^d}{ m}\,, 
$$
problems \eqref{pbcnoco} and \eqref{pbcbnoco} have a solution.
\end{theorem}

\begin{proof}
We  only consider \eqref{pbcnoco}, since the proof  of \eqref{pbcbnoco} is identical. Assume first that $m=\omega_d$.

As noticed in Theorem \ref{mainexistenceintro},
for every minimizing sequence $E_n\in \Kd$, with $|E_n|=\omega_d$, we can assume that there holds
\[  
\dP(E_n)\le Q^2 \Q( B).
\]
Thus, up to translating the sets $E_n$, by the quantitative isoperimetric inequality \cite{FMP} we can assume that
\[
|B\Delta E_n|^2\le C(d) \, \dP(E_n)\le C(d) Q^2 \Q(B)
\]
so that $|E_n\cap B^c|\le C Q$. Since every connected component of $E_n\in \Kd$ has volume at least 
$|B_\delta|=\omega_d\delta^d$,
for $Q\le c(\a,d)\delta^d$ the set $E_n$ must be connected. 
The existence of
minimizers then follows as in Theorem \ref{mainexistenceintro}. 

The case of a general volume $m$ can be obtain by rescaling from \eqref{scaleF}. 
\end{proof}


\noi It is natural to expect that, for  a charge $Q$ large enough, it is more favorable to have two connected
components rather than one, which would lead to non-existence of minimizers in $\Kd$. 
Let us prove that it is indeed the case, at least for small enough  $\a$. We start with the following lemma.

\begin{lemma}\label{stimenerdiam}
Let $\alpha>0$ and let $E$ be a compact set then
\[\Q(E)\ge \frac{1}{\diam(E)^\a}\,.\]
In particular,
\begin{equation}\label{estimenerpbc}
\inf_{|E|=\omega_d,E\in\Kdco}\F(E)
\ge \left(\frac{m}{\omega_d}\right)^\frac{d-1}{d}\P(B)+\left( \sqrt{d} \,
2^{d+2}\right)^{-\a}\,
Q^2 \delta^{(d-1)\alpha}\,,
\end{equation}
and 
\begin{equation}\label{estimenerpbcb}
\inf_{|E|=\omega_d,E\in\Kdco}\F(E)
\ge \left(\frac{m}{\omega_d}\right)^\frac{d-1}{d}\P(B)+\left( \sqrt{d} \,
2^{d+2}\right)^{-\a}\,  Q^2\delta^{(d-1)\alpha}.
 \end{equation}
\end{lemma}
\begin{proof}
 Let $\mu$ be any positive measure with support in $\ovE$ such that $\mu(E)=1$ then
\[\Q(E)\ge \int_{E\times E} \frac{d \mu(x) d\mu(y)}{|x-y|^\a}\ge \int_{E\times E} \frac{d \mu(x)
d\mu(y)}{\diam(E)^\a}=\frac{1}{\diam(E)^\a}.\]
By Lemma \ref{bound} and thanks to the isoperimetric inequality, we get \eqref{estimenerpbc} and
\eqref{estimenerpbcb}. 
\end{proof}

\noi We can now prove a non-existence result in $\Kd$. 

\begin{theorem}\label{nonexistencea<1}
For all $\a<1$ there exist $c_0=c_0(\a)>0$ and $Q_0=Q_0(\a)>0$ such that, for every $\delta>0$, $m\ge c_0\delta^d$, and
$$
\frac{Q}{m^\frac{d-1+\a}{2d}}\,>\, Q_0\, \left( \frac{m}{\delta^d}\right)^\frac{d\a+1-\a}{2d}
$$
problems \eqref{pbcnoco} and \eqref{pbcbnoco} do not have a solution.
\end{theorem}

\begin{proof}
We only discuss problem \eqref{pbcnoco}, since the non-existence result for problem \eqref{pbcbnoco} follows analogously.

As in Theorem \ref{mainexistence3intro} we first consider the case $m=\omega_d$, so that $\delta\le 1$.
If there exists a minimizer then the optimal measure $\mu$ is necessarily contained in a connected component of the minimizer. 
{}From \eqref{estimenerpbcb}
it then follows that the energy of the minimizer is greater than
\begin{equation}\label{low}
\P(B)+\left( \sqrt{d} \, 2^{d+2}\right)^{-\a}\, \delta^{(d-1)\alpha}Q^2\,, 
\end{equation}
which bounds from below the energy of any set in $\Kdco$ with volume $\omega_d$.
Hence, in order to prove the  non-existence, it is enough to construct a competitor $E\in \Kd$ with energy less than
\eqref{low}.

Consider the set $E$ given by  $N$ (which we suppose to be an integer)
balls of radius $\delta$, equally charged. Up to increasing their
mutual distances, we can suppose that the Riesz potential energy of $E$ is made only of the self interaction of
each ball with itself. Since $N=\delta^{-1}$ we then have
\begin{equation}\label{stimapd}
\P(E)+Q^2 \Q( E)=N \delta^{d-1} \P(B) + \frac{Q^2}{N} \Q( B_\delta)
= \frac{1}{\delta} \P(B)  + \Q(B)\delta^{d-\a}Q^2.
\end{equation}
Notice that, if $d-\a>(d-1)\a$, i.e. if $\a<1$, there exists $\delta_0=\delta_0(\a)$ such that
for all $\delta\le\delta_0$ there holds  
$$\Q( B)\,\delta^{d-\a}\le \frac 1 2 \left( \sqrt{d} \, 2^{d+2}\right)^{-\a}\delta^{(d-1)\alpha}.$$ 
With this condition in force, from \eqref{stimapd} we get 
\[\P(E)+Q^2 \Q( E)< \P(B)+\left( \sqrt{d} \, 2^{d+2}\right)^{-\a}\, Q^2\delta^{(d-1)\alpha}\,,\]
for 
\[
Q>\sqrt{2\P(B)}\left( \sqrt{d} \, 2^{d+2}\right)^\frac{\a}{2}\frac{1}{\delta^\frac{d\a+1-\a}{2}}\,.
\]
The general case can be obtain by rescaling from \eqref{scaleF}. 
\end{proof}

\begin{remark}\rm
 If $\a< \frac{d-1}{d}$, we can improve the previous estimate on $Q$ by considering a construction similar to the one of Theorem \ref{nonexistR}. 
 Indeed, for $\beta\in (d\a,d-1)$, taking $N:= \delta^{-{\beta}}$  charged balls
of radius $\delta$ and a non charged ball of volume $m-\omega_d N \delta^d$, we find a contradiction if  
\[  \frac{Q}{m^\frac{d-1+\a}{2d}}\,>\, \widetilde Q_0(\a)\, \left( \frac{m}{\delta^d}\right)^\frac{{\beta}-(1-\a)(d-1)}{2d}.\]
Notice that, if $\a<\frac{d-1}{2d-1}$, we can choose $\beta$ such that the exponent $\frac{{\beta}-(1-\a)(d-1)}{2d}$ is negative.
\end{remark}
\begin{remark}\rm
We expect that the non-existence result in Theorem \ref{nonexistencea<1}
also holds for $\alpha\geq 1$, but we where unable
to show this, as the class $\Kd$ is fairly rigid which makes the construction of competitors quite delicate.
\end{remark}

\section{Minimality of the ball}\label{droplets5}

\noi In this section we prove that, in the harmonic case $\a=d-2$,  the
ball is a minimizer for Problem \eqref{pbcnoco} (for $\Omega=\R^d$) among sets in the family of the {\it nearly spherical sets}
belonging
to $\Kdco$ introduced in Definition \ref{d-ball}, that is, the sets which are a small $W^{1,\infty}$
perturbation of the ball and that satisfy the $\delta$-ball condition.

\noi Consider a set $E$ such that $|E|=\omega_d$, and such that  $\partial E$ can be written as a graph over $\partial B$.
In polar coordinates we have
\[
E=\big\{R(x)x \, :\, R(x)=1+\p(x),\, x\in\partial B\big\}.
\] 
The condition $|E|=\omega_d$ then becomes
\[
\int_{\pa B}\left((1+\p(x))^d - 1\right) d\Hdm(x)= 0
\]
which implies that if $\|\vphi\|_{L^\infty(\pa B)}$ is small enough, then
\begin{equation}\label{E=B}
\int_{\pa B}\p d\Hdm=O(\|\p\|_{L^2(\pa B)}^2).
\end{equation}
Letting 
\[
\bar\phi := \frac{1}{|\partial B|}\int_{\pa B}\p d\Hdm\,,
\]
the Poincar\'e Inequality gives
\begin{equation}\label{poincare}
\begin{aligned}
\int_{\pa B}|\nabla\p|^2 d\Hdm &\ge C \int_{\pa B}|\p-\bar\p|^2 d\Hdm= C(d) \int_{\pa
B}\p^2\Hdm-\frac{C(d)}{d\omega_d}\left(\int_{\pa B}\p d\Hdm\right)^2
\\
&= C(d) \int_{\pa B}\p^2 d\Hdm- \frac{C}{4d\omega_d}\left(\int_{\pa B}\p^2 d\Hdm\right)^2
\\
&\ge\frac 3 4 C(d) \int_{\pa B}\p^2 d\Hdm
\end{aligned}
\end{equation}
as soon as 
\begin{equation}\label{condpdue}
\int_{\pa B}\p^2 d\Hdm\le d\omega_d.
\end{equation}

\noi Up to  translation, we can also assume that the barycenter of $E$ is  $0$. 
This implies that
\begin{equation}\label{barE=0}
\left|\int_{\partial B} x\vphi(x) d\Hdm(x) \right|\ =\ O\left(\|\p\|_{L^2(\pa B)}^2\right)\,.
\end{equation}

\begin{lemma}\label{lemmapoinc}
Suppose that $\vphi:\partial B\to\R^d$ parametrizes $\partial E$ and  $\|\vphi\|_{L^\infty(\pa B)}$ is small enough
so that \eqref{condpdue} is satisfied. Assume also that the barycenter of $E$ is in $0$.
 Then,
\begin{equation}\label{controllodp}
  \dP(E)\ge c_0 \int_{\partial B}|\nabla\phi|^2 d\Hdm\ge c_1 \int_{\partial B}|\phi|^2 d\Hdm=\frac{c_1}{2}
\left|\int_{\partial B}\phi d\Hdm\right| \,.
\end{equation}
\end{lemma}

\begin{proof}
 We refer to \cite{fuglede} for the  proof of the first inequality. The
second inequality is  \eqref{poincare}, while the third one follows from \eqref{E=B}.
\end{proof}
\noi A consequence of Lemma \ref{lemmapoinc}  is the following corollary.

\begin{corollary}\label{Qephi}
Suppose that $\partial E$ is parametrized on $\partial B$ by a function $\phi$ which satisfies the hypothesis of
Lemma
\ref{lemmapoinc}. Then there exists a positive constant
$C=C(\a,d)$ such that
\begin{equation}\label{q1}
 |\I_\a^{\partial B}(\phi)|\le C \,   \dP(E),
\end{equation}
and, for any positive constant $\lambda$,
\begin{equation}\label{q2}
 |\I_\a^{\partial B}(\lambda,\phi)|\le C \lambda\,   \dP(E).
\end{equation}
\end{corollary}
\begin{proof}
 Inequality \eqref{q2} is an immediate consequence of \eqref{controllodp}. Concerning the first one we have, by
the H\"older inequality and the Fubini Theorem,
 \[
 \begin{aligned}
 \I_\a^{\partial B}(\phi)&=\int_{\pa B\times\pa B}\frac{\phi(x)\phi(y)}{|x-y|^\a} \, d \Hdm(x)d \Hdm(y)\\
 &\le \left(\int_{\pa B\times\pa B}\frac{\phi(x)^2}{|x-y|^\a}\, d \Hdm(x)d \Hdm(y)\right)^{1/2}
 \left(\int_{\pa B\times\pa
B}\frac{\phi(y)^2}{|x-y|^\a}\, d \Hdm(x)d \Hdm(y)\right)^{1/2}\\
&= C\int_{\pa B}\phi(x)^2\,d \Hdm(x).
 \end{aligned}
 \]
So \eqref{q1} follows again from \eqref{controllodp}. 
\end{proof}

We will use the following technical lemma.

\begin{lemma}\label{shape}
Let $E=\big\{R(x)x \, :\, R(x)=1+\p(x),\, x\in\partial B\big\}$ and let $g\in L^\infty(\partial B)$, then there
exists $\eps_0(\alpha,d)$ and
  a constant $C=C(\a,d)>0$ such that if $\WinfB{\p}\le \eps_0\le1$, 
\begin{equation}\label{stab2}
\begin{aligned}
&\left|\int_{ \partial B\times\partial  B} \left(\frac{1}{|R(x)-R(y)|^\a}-
\frac{(1-\frac\a2\phi(x))(1-\frac\a2\phi(y))}{|x-y|^\a} \right) g(x) g(y)d \Hdm(x)\,d \Hdm(y)\right| \\
&\qquad \le	C(\a,d)(1+\varepsilon_0)\LinfB{g}^2\,   \dP(E).
\end{aligned}
\end{equation}
\end{lemma}

\begin{proof}
 
First, notice that since $|x|=|y|=1$ we have
\begin{equation}\label{stab3} 
|R(x)x-R(y)y|^2=|x-y|^2\left(1+\phi(x)+\phi(y)+\phi(x)\phi(y)+\psi(x,y)\right)
\end{equation}
where $\psi(x,y)=\frac{(\phi(x)-\phi(y))^2}{|x-y|^2}$. Hence, for any $x,y\in\partial B$ there holds,
\begin{equation}\label{stab4}
\begin{aligned}
|R(x)x-R(y)y|^{-\a}&=\frac{(1-\frac{\a}{2}\phi(x))(1-\frac{\a}{2}\phi(y))+\frac{\a(4-\a)}{4}
\phi(x)\phi(y)-\frac{\a}{2}(\psi(x,y)+\eta(x,y))}{|x-y|^{\a}} \\
\end{aligned}
\end{equation}
where
\[
0\le \eta(x,y)\le C\left(\phi^2(x)+\phi^2(y)+\psi^2(x,y)\right).
\]
By \eqref{stab4} we get

\begin{equation}\label{stab5}
\begin{aligned}
 &\int_{ \partial B\times\partial  B}\left(\frac{1}{|R(x)-R(y)|^\a}-
\frac{(1-\frac\a2\phi(x))(1-\frac\a2\phi(y))}{|x-y|^\a} \right) g(x) g(y)d \Hdm(x)\,d \Hdm(y)\\
&=\frac{\a(4-\a)}{4}\int_{\pa B\times\partial
B}\frac{\vphi(x)\vphi(y)}{|x-y|^\a} g(x) g(y)\,d \Hdm(x)d \Hdm(y)\\
&\quad -\frac{\a}{2}\int_{\partial B\times\pa B}\frac{\psi(x,y) + \eta(x,y)}{|x-y|^\a} g(x) g(y)\,d \Hdm(x) d
\Hdm(y).
\end{aligned}
\end{equation}
 By Corollary \ref{Qephi} we get
\[
\int_{\pa B\times\partial
B}\frac{\vphi(x)\vphi(y)}{|x-y|^\a}\,d \Hdm(x)d \Hdm(y)=  \I_\a^{\partial B}(\phi)\le C  \dP(E).
\]
Furthermore, we have 
\[
0\le \psi(x,y)\le \LinfB{\nabla \phi}^2\le\varepsilon_0, 
\]
and
\[
 \int_{\partial B\times\partial B}\frac{\varphi(x)^2\,d\H^{d-1}(x)d\H^{d-1}(y)}{|x-y|^\a}=\int_{\partial
B}\frac{d\H^{d-1}(y)}{|x-y|^\a}\int_{\partial
B}\varphi(x)^2\,d\H^{d-1}(x)\le c(\a,d)\varepsilon_0^2,
\]
for a suitable constant $c(\a,d)$. Therefore, since $\eta(x,y)\le  C\left( \phi^2(x)+\phi^2(y)+\psi(x,y)\right)$, to
prove
\eqref{stab2} we only have to check that
\[
\int_{\partial B\times\pa B} \frac{\psi(x,y)}{|x-y|^\alpha} \,d \Hdm(x)d \Hdm(y)\le C \dP(E).
\]

\noi To this aim, consider $x,y$ in $\partial B$ and denote by $\Gamma_{x,y}$ the geodesic going from $x$ to
$y$ and by
$\ell(x,y)$ the
geodesic distance between $x$ and $y$ (that is the length of $\Gamma_{x,y}$).
Notice that on $\partial B$, the euclidean distance and $\ell$ are equivalent so that it is enough proving
\[
\int_{\partial B\times\pa B} \ell(x,y)^{-(\a+2)}(\phi(x)-\phi(y))^2\,d\H^{d-1}(x)d\H^{d-1}(y)\le C \dP(E).
\]
We have
\begin{align*}
\int_{\partial B\times\pa B} &\ell(x,y)^{-(\a+2)}(\phi(x)-\phi(y))^2\\
&\le c(d) \int_{\partial B\times\pa B}
\ell(x,y)^{-(\a+1)} \int_{\Gamma_{x,y}} |\nabla \phi|^2(z) dz \ d \Hdm(x) d \Hdm(y)\\
&\le c(d)\int_{\partial B}\int_0^{2\pi} t^{-(\a+1)}t^{d-1}\left(\int_{\{\ell(x,z)\le t\}} |\nabla \phi|^2 (z)
d\Hdm(z)\right) dt\, d \Hdm(x)\\
&=c(d)\int_0^{2\pi} t^{(d-1)-(\a+1)}\left(\int_{\partial B}\int_{\{\ell(x,z)\le t\}} |\nabla \phi|^2 (z) d \Hdm(x)
d\Hdm(z)\right) dt\\
&=c(d)\mathcal{H}^{d-2}(\mathbb S^{d-2})\int_0^{2\pi} t^{(d-1)-\a}\left(\int_{\partial B} |\nabla \phi|^2 (z)
d\Hdm(z)\right) dt\\
&=c(d)\mathcal{H}^{d-2}(\mathbb S^{d-2})\int_0^{2\pi} t^{(d-1)-\a} dt \left(\int_{\partial B} |\nabla \phi|^2 (z)
d\Hdm(z)\right)\\
&\le 
C\dP(E)
\end{align*}
where $\mathbb S^{d-2}$ is the $(d-2)$-dimensional sphere and where we used the fact that $\a<d-1$ together with
\eqref{controllodp}.
\end{proof}

\noi Before we  prove our main stability estimates, we recall a classical
interpolation inequality.

\begin{lemma}
 For every $0\le p<q<r<+\infty$, there exists a constant $C(r,p,q)$ such that for every $\p\in H^{r}(\R^d)$, 
 there holds
\begin{equation}\label{interpol}
 \|\p\|_{\stackrel{}{H^q(\R^d)}}\le C  \left( \|\p\|_{\stackrel{}{H^r(\R^d)}}\right)^{\frac{r-q}{r-p}}\left(
\|\p\|_{\stackrel{}{H^p(\R^d)}}\right)^{\frac{q-p}{r-p}},
\end{equation}
where we adopted the notation $\|u\|_{H^p(\R^d)}:=\||\xi|^p\hat
u\|_{L^2(\R^d)}$
and
$H^{p}(\R^d):=\{u\in L^2(\R^d):\|u\|_{H^p}<+\infty\}$, being $\hat u$ the Fourier transform of the function $u$. 
\end{lemma}
\begin{proof}
Let $\p\in H^r(\R^d)$ and $\lambda>0$, then we have
\begin{align*}
  \|\p\|^2_{\stackrel{}{H^q(\R^d)}}&=\int_{\R^d} |\hat \p|^2 |\xi|^{2q} d\xi=\int_{|\xi|\le \lambda} |\hat \p|^2
|\xi|^{2p} |\xi|^{2(q-p)}d\xi+\int_{|\xi|\ge \lambda} |\hat \p|^2 |\xi|^{2r} |\xi|^{2(q-r)}d\xi\\
&\le \lambda^{2(q-p)} \|\p\|^2_{\stackrel{}{H^p(\R^d)}}+ \lambda^{-2(r-q)} \|\p\|^2_{\stackrel{}{H^r(\R^d)}}.
\end{align*}
An optimization in $\lambda$ yields \eqref{interpol}.
\end{proof}

\begin{proposition}\label{mainstab}
Let $\a\in [d-2,d-1)$, $f\in L^\infty(\pa E)$ and 
$$
\partial E=\big\{R(x)x \, :\, R(x)=1+\p(x),\, x\in\partial B\big\}.
$$ 
Then there exist $\eps_0(\alpha)>0$ and $C=C(\a)>0$ such that if $\WinfB{\p}\le \eps_0$  then
\begin{equation}\label{mainstabequation}
\I_\a^{\partial E}(f)-\I_\a^{\partial B}(\bar f)
    \ge -C\|f\|^2_{L^\infty(\partial E)}  \dP(E),
\end{equation}
where $\bar f:= \ds\frac{1}{\P(E)}\int_{\partial E} f d\Hdm$.
\end{proposition}

\begin{proof}
We have
\begin{equation}\label{stima1}
\begin{aligned}
\I_\a^{\partial E}(f) &=\int_{ \partial E\times\partial  E}\frac{f(x)f(y)}{|x-y|^\a} \,d \Hdm(x)\,d \Hdm(y)\\
&=\int_{\partial
B\times \partial B}\frac{{g}(x){g}(y)}{|R(x)-R(y)|^\a}\,d \Hdm(x)d \Hdm(y)
\end{aligned}
\end{equation}
where we set 
\[
 g(x)=f(R(x)x)R(x)^{d-2}\sqrt{R(x)^2+|\nabla R(x)|^2}.
 \]

Up to choose $\varepsilon_0$ small enough, we can suppose that 
\begin{equation}\label{fgeg}
\|g\|_{L^\infty(\partial B)}\le 2 \|f\|_{L^\infty(\partial E)}.
\end{equation}
Let $\bar g:=\ds\frac{1}{\P(B)}\int_{\partial B} g d\Hdm=\frac{P(E)}{P(B)} \bar f$. Then we have
\[
 \Q^{\partial E}(f)-\Q^{\partial B}(\ov f)= \Q^{\partial E}(f)-\Q^{\partial B}(\ov g)+\Q^{\partial B}(\ov
g)-\Q^{\partial B}(\ov f).
\]
Focusing on the last two terms in the previous equality we have
\[
 \begin{aligned}
  \left|\Q^{\partial B}(\ov g)-\Q^{\partial B}(\ov f)\right|&=\Q^{\partial B}(\ov f)
\left|1-\left(\frac{\P(E)}{\P(B)}\right)^2\right|\\
&=C {\bar f}^2 \frac{\P(E)+\P(B)}{\P(B)^2} |\P(E)-\P(B)|\\
&\le  C(\a,d)\|f\|^2_{L^\infty(\partial E)}\dP(E).
 \end{aligned}
\]
Therefore, to prove \eqref{mainstabequation} we only need to show that

\begin{equation}\label{tesi2}
 \I_\a^{\partial E}(f)\ge \I_\a^{\partial B}(\bar g) - \|g\|^2_{L^\infty(\partial B)} \,   \dP(E).
\end{equation}

\noi Formula \eqref{stima1} together with Lemma \ref{shape} imply
\[
\I_\a^{\partial E}(f)=\I_\a^{\partial B}\left(g(1-\frac\a2\phi)\right)+\mathcal R(g,\phi)
\]
with 
\[
|\mathcal R(g,\phi)|
\le c\| g\|_{L^\infty(\partial E)}^2\,   \dP(E),
\]
so that 
\begin{equation}\label{q1b}
\I_\a^{\partial E}(f)\ge \I_\a^{\partial B}\left({g}(1-\frac\a2\phi)\right)-c \| g\|_{L^\infty(\partial
E)}^2\,   \dP(E).
\end{equation}

We need to estimate $\Q^{\partial B}(g(1-\a/2)\varphi)$. By the bilinearity of $\Q^{\partial B}$ we have that

\begin{equation}\label{contone}
\begin{aligned}
\I_\a^{\partial B}({g}(1-\frac\a2\phi))=&\I_\a^{\partial B}({g}(1-\frac\a2\phi),{g}(1-\frac\a2\phi))
\\
=&\I_\a^{\partial B}({g},{g})-\a\I_\a^{\partial B}({g},{g}\phi)
+\frac{\a^2}{4}\I_\a^{\partial B}({g}\phi,{g}\phi)
\\
=&\I_\a^{\partial B}(\bar {g},\bar {g})+\I_\a^{\partial B}({g}-\bar {g},{g}-\bar {g})-\a\I_\a^{\partial B}({g}-\bar
{g},{g}
\phi)
-\a\I_\a^{\partial B}(\bar {g}, {g} \phi)
\\
&+\frac{\a^2}{4}\I_\a^{\partial B}(\bar {g}\phi,\bar {g}\phi)+\frac{\a^2}{2}\I_\a^{\partial B}(\bar
{g}\phi,({g}-\bar
{g})\phi)
+\frac{\a^2}{4}\I_\a^{\partial B}(({g}-\bar {g})\phi,({g}-\bar {g})\phi)
\\
=&\I_\a^{\partial B}(\bar {g})+\I_\a^{\partial B}({g}-\bar {g})
+\frac{\a^2}{4}\I_\a^{\partial B}(({g}-\bar {g})\phi)
-\a\I_\a^{\partial B}({g}-\bar {g},({g}-\bar {g})\phi)\\
&-\a\I_\a^{\partial B}(\bar {g}, ({g}-\bar {g}) \phi)
-\a\I_\a^{\partial B}({g}-\bar {g},\bar {g}\phi)
+\frac{\a^2}{2}\I_\a^{\partial B}(\bar {g}\phi,({g}-\bar {g})\phi)
\\
&-\a\I_\a^{\partial B}(\bar {g}, \bar {g} \phi)
+\frac{\a^2}{4}\I_\a^{\partial B}(\bar {g}\phi).
\end{aligned}
\end{equation}
Thanks to \eqref{q2}, the last two terms in the right hand side of \eqref{contone} satisfy:
\begin{equation}\label{c1}
 -\I_\a^{\partial B}(\bar {g}, \bar {g}\vphi)
+\frac{\a}{4}\I_\a^{\partial B}(\bar {g}\phi)\ge-c\bar g^2 \,   \dP(E).
 \end{equation}
By the Cauchy-Schwarz inequality \eqref{cs} and Young's inequality, we get that  for every functions $h_1$ and $h_2$ and for any
$\eps>0$,   
\begin{equation}\label{young}
\I_\a^{\partial B}(h_1,h_2)\le \I_\a^{\partial B}(h_1)^\frac 1 2 \I_\a^{\partial B}(h_2)^\frac 1 2\le
\eps\I_\a^{\partial B}(h_1)+\frac{1}{4\eps}\I_\a^{\partial B}(h_2).
\end{equation}
In particular, applying such inequality to the functions $h_1=g-\bar g$ and $h_2=(g-\bar g)\vphi$ in the fourth
term in the right hand side of \eqref{contone}, and then to  $h_1=g-\bar g$ and $h_2=\bar g\vphi$ in the
sixth term,
and exploiting \eqref{c1}, we  obtain the existence of a positive constant $C$ such that
\begin{equation}\label{conto}
 \begin{aligned}
&\I_\a^{\partial B}({g}(1-\frac\a2\phi))-\I_\a^{\partial B}(\bar g)\\
&\ge C\left(\frac{1}{2}\I_\a^{\partial B}(g-\bar g) -  \I_\a^{\partial B}(\bar {g}, ({g}-\bar{g})
\phi)-\I_\a^{\partial B}
((g-\bar g)\vphi)-\bar {g}^2 \,   \dP(E)\right).
\end{aligned}
\end{equation}
Again, by Lemma \ref{lemmapoinc}, we have that
\[
 -\I_\a^{\partial B}((g-\bar g)\vphi)\ge -\|g\|_{L^\infty(\partial
B)}^2\I_\a^{\partial B}(\vphi)\ge-C\|g\|_{L^\infty(\partial B)}^2\,   \dP(E).
\]

\noi Let us show that the term $\I_\a^{\partial B}(\bar {g}, ({g}-\bar {g}) \phi)$ 
can be estimated by the term $\I_\a^{\partial B}({g}-\bar {g})$.

\noi Let $\widetilde\phi:\R^d\to \R$ be a regular extension of $\phi$, and let $\widetilde g= ({g}-\bar
{g})d\mathcal
H^{d-1}\res{\partial B}$.
By a Fourier transform we get
\begin{equation*}
\begin{aligned}
\I_\a^{\partial B}(\bar {g}, ({g}-\bar {g}) \phi) =& \int_{\partial B}\frac{d\mathcal H^{d-1}(x)}{|x-y|^\a}\bar
{g}\int_{\partial B}({g}-\bar {g})\,d\mathcal H^{d-1}(y)\phi
= c(\a,d) \bar {g}\int_{\R^d}\widehat{\widetilde\phi}\widehat{\widetilde g}
\\
\le& \bar {g} \left( \int_{\R^d} \widehat{\widetilde\vphi}^2|\xi|^{d-\alpha}\right)^\frac 1 2 
\left(\int_{\R^d}\frac{\widehat{\widetilde g}^2}{|\xi|^{d-\alpha}}\right)^\frac 1 2
\\
=&   \bar {g} \|{\widetilde\phi}\|_{H^\frac{d-\alpha}{2}(\R^d)}
\I_\a^{\partial B}({g}-\bar {g},{g}-\bar {g})^\frac{1}{2}
\\
\le& C(d) \bar {g} \|\phi\|_{H^\frac{d-\alpha}{2}(\partial B)}
\I_\a^{\partial B}({g}-\bar {g})^\frac{1}{2}. 
\end{aligned}
\end{equation*}

\noi We  now observe that, if 
\begin{equation}\label{esso}
\I_\a^{\partial B}(\bar {g}, ({g}-\bar {g}) \phi)\le 
\frac 1 {2} \I_\a^{\partial B}({g}-\bar {g}),
\end{equation}
then we would get
\[
\I_\a^{\partial B}({g}(1-\frac\a2\phi))-\I_\a^{\partial B}(\bar {g})
\ge -C\|\bar {g}\|_{L^\infty(\partial B)}^2 \,   \dP(E),
\]
which would imply \eqref{tesi2} and so the claim of the proposition. On the other hand if \eqref{esso} does not
hold, then, up to consider again a regular extension $\tilde \varphi:\R^d\to\R$ of $\varphi$, we have 
\[
\I_\a^{\partial B}({g}-\bar {g})<C(d)\bar {g} \|\phi\|_{H^\frac{d-\alpha}{2}(\partial B)}
\I_\a^{\partial B}({g}-\bar {g})^\frac{1}{2},
\]
which implies
\[
\I_\a^{\partial B}({g}-\bar {g})^\frac 1 2<C\bar {g} \|\phi\|_{H^\frac{d-\alpha}{2}(\partial B)},
\]
so that
\[
\I_\a^{\partial B}(\bar {g}, ({g}-\bar {g}) \phi)\le 
C \bar {g} \|{\phi}\|_{H^\frac{d-\alpha}{2}(\partial B)}
\I_\a^{\partial B}({g}-\bar {g})^\frac{1}{2} 
\le C \bar {g}^2\|\phi\|_{H^\frac{d-\alpha}{2}(\partial B)}^2.
\]
If $\frac{d-\alpha}{2}\le1$ then using \eqref{interpol} with $p=0$, $q=\frac{d-\alpha}{2}$ and $r=1$, up to once
again regularly extend $\varphi$ on $\R^d$, we obtain
\[\|\phi\|_{H^\frac{d-\alpha}{2}(\partial B)}^2\le c_0
\left(
\|\phi\|_{\stackrel{}{H^1(\partial B)}}^2\right)^{1-\frac{d-\a}{2}}\left(\|\phi\|_{\stackrel{}{L^2(\partial
B)}}^2\right)^{\frac{d-\a}{2}}
\le c_1\left( \|\phi\|_{\stackrel{}{H^1(\partial B)}}^2+ \|\phi\|_{\stackrel{}{L^2(\partial B)}}^2\right)\le C  
\dP(E),
\]

 which  concludes the proof.
\end{proof}

\begin{theorem}\label{stabilitycore}
Let $d\ge 3$ and $\a=d-2$. Then for any $\delta>0$ and $m\ge \omega_d \delta^d$, there exists a charge $\bar Q
\lt(\frac{\delta}{m^{1/d}}\rt)>0$,  such that if 
$$\frac{Q}{m^\frac{d-1+\a}{2d}}\,\le\,\bar Q \lt(\frac{\delta}{m^{1/d}}\rt)
$$
the ball is the
unique minimizer of problem \eqref{pbcnoco}.  
\end{theorem}
\begin{proof}
Up to a rescaling we can assume $m=\omega_d$. 
By Theorem \ref{mainexistence3intro}, there exists $C>0$ such that 
problem \eqref{pbcnoco} admits a minimizer $E_Q$ for every $Q\in (0,C\delta^\frac d2)$.
Since $|E_Q\Delta B|^2\le C \dP(E_Q)\le Q^2 \Q(B)$, $E_Q$ 
converges to $B$ in $L^1$ when $Q\to 0$. As in Theorem \ref{mainexistenceintro}, there is also convergence in the
Hausdorff sense
 of $E_Q$ and $\partial E_Q$ thanks to the $\delta$-ball condition. Again, by the $\delta$-ball condition and
the
Hausdorff convergence of the boundaries, 
for $Q$ small enough, $\partial E_Q$ is a graph over $\partial B$ of some $C^{1,1}$ function with $C^{1,1}$ norm
bounded by $2/\delta$. 
>From this we see that if $\partial E_Q=\{(1+\vphi_Q(x))x \ :\ x\in\partial B\} $ then
$\WinfB{\p_Q}$ is converging to $0$. We can thus assume that $\vphi_Q$ satisfies
 the hypotheses of Proposition \ref{mainstab}.

\noi Let $\mu=f d\Hdm\res \partial
E_Q$ be the minimizer of $\Q(E_Q)$.
 Since $\Q(E_Q)\le \P(B)+Q^2\Q(B)$, by Proposition \ref{regolarita}, $\LinfE{f}\le (d-2)\delta^{-1}
(\P(B)+Q^2\Q(B))$. 
Let $ \ds \bar f:=\frac{1}{\P(E_Q)}=\frac{1}{P(E_Q)}\int_{\partial E_Q} f d\Hdm$.
By Lemma \ref{ball} we know that the optimal measure for $\Q(B)$ is given by  $\frac{\Hdm\res\partial B}{\P(B)}$.
By the minimality of $E_Q$ we then have
\[
\begin{aligned}
      \dP(E_Q)&=\P(E_Q)-\P(B)\le Q^2(\Q(B)-\Q(E_Q))\\
    &=Q^2\left(\I_\a^{\partial B}(\bar f)-\I_\a^{\partial E_Q}(f)+\I_\a^{\partial B}(1/\P(B))-\I_\a^{\partial
B}(1/\P(E_Q))\right).
\end{aligned}
\]
A simple computation shows that
\[
 \I_\a^{\partial B}(1/\P(B))-\I_\a^{\partial B}(1/\P(E_Q))\le C^2\,   \dP(E_Q)
\]
for a suitable positive constant $C=C(\a,d)$. Hence, by Proposition \ref{mainstab} we have that
\[
   \dP(E_Q)\le C Q^2\,   \dP(E_Q)(1+\|f\|_{L^\infty(\partial E_Q)}^2) \le C Q^2\,   \dP(E_Q),
\]
which implies $\dP(E_Q)=0$ that is $E_Q=B$, for $Q$ small enough.
\end{proof}

\begin{remark}\rm
 We recall that a counterpart of Theorem \ref{stability} holds as well. Indeed, in \cite{fontfried} it was proven
that if $Q$ overcome a certain threshold, any radial set (and in particular the ball) is unstable under small $C^{2,\beta}$ perturbations.
\end{remark}

\begin{remark}\rm
The previous proof of the stability does not apply to the case $\a>d-2$. Indeed, this proof relies on $L^\infty$
bounds for the optimal measure $\mu$ for $\Q$ which we are not able to obtain in that case. 
For the very same reason,  our approach seems not to work if we replace the class $\Kd$ by the class 
of convex sets. In fact, for a set with Lipschitz boundary, the optimal measure is not  expected to be in
$L^\infty$. In particular, if $E$ is convex, than its optimal measure  blows-up at every non-regular point of
$\partial E$, as shown in Example \ref{acuti}. 
\end{remark}

\section{The logarithmic potential energy}\label{droplets6}
\noi In this section we investigate the same type of questions   for the logarithmic potential which is given 
by $-\log(|x|)$. This potential naturally arises in two dimension where it corresponds to the Coulomb interaction. Let then
\begin{equation}\label{Glog}
\Il( E):=\min_{\mu(\ovE)=1}\ \int_{\R^d\times\R^d} -\log(|x-y|)d\mu(x) d\mu(y)\end{equation} 
and consider the problem
\begin{equation}\label{problog}\min_{|E|=m} \P(E)+Q^2 \Il( E).
\end{equation}
In  analogy to the notation adopted for the Riesz potential we define, for any Borel functions $f$ and $g$, the
following quantity
\[
 \Il^{\partial E}(f,g):=\int_{\partial E\times\partial E}-\log(|x-y|)f(x)\, g(y)d\Hdm(x) \, d\Hdm(y).
\]
We list below some important
properties of $\Il$ without  proof, since they are analogous to those given in Section
\ref{droplets2} for the Riesz potential. We refer to 
\cite{safftotik,landkof} for comprehensive guides  on  the logarithmic potential.

\begin{proposition}
 The following properties hold:
\begin{itemize}
 \item [(i)] for every compact set $E$, there exists a unique optimal measure $\mu$ for $\Il(E)$ which is
concentrated on the boundary of $E$,
\item [(ii)] for every Borel measure $\mu$ it holds 
\[
\Il(\mu)=\int_{\R^d} \left(v^\mu_{d/2}(x)\right)^2 dx\ge 0
\]
 where 
 \[
v_{d/2}^\mu(x)=\int_{\R^d}-\log|x-y|\,d\mu(y),
\]
\item [(iii)] for every smooth set $E$, if $\mu$ is the optimal measure for $\Il$, then the equality
$\displaystyle \int_{ \partial E} -\log(|x-y|) d\mu(y) =\Il(E)$ holds for every
$x\in \partial E$. Moreover the optimal measure for the ball is the uniform measure,
\item [(iv)] if $d=2$, then for every bounded set $E$  satisfying the $\delta$-ball condition, the optimal
measure
is given by some measure $\mu= f \Hdm\res \partial E$ with $\LinfE{f}\le \frac{\Il(E)}{|\log(\delta)|}$.
\end{itemize}

\end{proposition}

 \noi In this setting, since the potential can be negative, the picture  is slightly
different from that related to the Riesz energy. Indeed, we have the following Theorem.
\begin{theorem}
 The following statements hold true:
\begin{itemize}
 \item[(i)] $\inf_{|E|=m} \P(E)+Q^2 \Il( E)=-\infty$.
\item[(ii)] for any $\delta>0$, if $m>2\omega_d \delta^d$ then  $\inf_{|E|=m, E\in \Kd} \P(E)+Q^2 \Il( E)=-\infty$,
\item[(iii)] for every $Q>0$  and every $m> \omega_d \delta^d$, there exists a minimizer of 
\[
\min_{|E|=m, E\in \Kdco} \P(E)+Q^2 \Il( E),
\]
\item[(iv)] for every bounded smooth domain $\Om$, 
 \[
 \inf_{|E|=m, E\subset \Om} \P(E)+Q^2 \Il( E)=\min_{|E|=m, E\subset \Om} \P(E) +Q^2 \Il(\Om).
 \]
\end{itemize}
\end{theorem}
 
\begin{proof}
Statement $(ii)$ implies $(i)$ while $(iii)$ can be proven exactly as in Theorem \ref{mainexistenceintro}
and $(iv)$ as Theorem \ref{main}. To prove $(ii)$ we set $E_n=B_\delta(x_1^n)\cup B_\delta(x_2^n)$ and
notice that if ${\rm dist}(x_1^n,x_2^n)$ goes to infinity, then $\Il( E_n)\to -\infty$ as $n\to+\infty$.   
\end{proof}
\noi Since $\Il(\lambda
E)=\Il(E)-\log(\lambda)$ for every $\lambda>0$,  without loss of generality
we shall assume that  $m=|B_{1/2}|=\pi/4$ in Problem \eqref{problog}. 
The following result is the counterpart of Proposition \ref{mainstab}.

\begin{proposition}\label{stablog}
 Let $d=2$,  $E=\big\{R(x)x \, :\, R(x)=1+\p(x),\, x\in\partial B_{1/2}\big\}$ and let $f\in L^\infty(\pa E)$ then
there exists $\eps_0$ and
  a constant $C=C(\a)>0$ such that if $\WinfBd{\p}\le \eps_0$. Then
\[
\I_\lo^{\partial E}(f)-\I_\lo^{\partial B_{1/2}}(\bar f)
\ge  - C \|f\|^2_{L^\infty(\partial E)} \,   \dP(E),
\]
where $\bar f:=\displaystyle \frac{1}{\P(E)}\int_{\partial E} f d\H^1$.
\end{proposition}

\begin{proof}
 Notice that since $E\subset B$, the logarithmic potential is positive. As in the proof of Proposition \ref{mainstab}, we have
\[\I_\lo^{\partial E}(f)=\int_{\partial B_{1/2}\times \partial B_{1/2}} -\log(|R(x)-R(y)|) g(x) g(y) d\H^1(x) d\H^1(y),\]
where $g(x)=f(R(x)x)\sqrt{R(x)^2+|\nabla R(x)|^2}$. Reminding that from \eqref{stab3}, we have 
\[|R(x)x-R(y)y|=|x-y|\left(1+\phi(x)+\phi(y)+\phi(x)\phi(y)+\psi(x,y)\right)^{1/2},\]
where, $\psi(x,y)=\frac{(\phi(x)-\phi(y))^2}{|x-y|^2}$, we see that
\begin{multline*}\I_\lo^{\partial E}(f)=\int_{\partial B_{1/2}\times \partial B_{1/2}} -\log(|x-y|) \,g(x) g(y)\, d\H^1(x) d\H^1(y)\\
+\frac{1}{2} \int_{\partial B_{1/2}\times \partial B_{1/2}} -\log(1+\phi(x)+\phi(y)+\phi(x)\phi(y)+\psi(x,y))\, g(x) g(y)\, d\H^1(x) d\H^1(y).\end{multline*}
As in Proposition \ref{mainstab}, letting $\bar g:=\displaystyle \frac{1}{\P(B_{1/2})}\int_{\partial B_{1/2}} g \, d\H^1$, we have
 \[\I_\lo^{\partial B_{1/2}}(g)=\int_{\partial B_{1/2}\times \partial B_{1/2}} -\log(|x-y|) \,g(x) g(y)\, d\H^1(x) d\H^1(y)=\I_\lo^{\partial B_{1/2}}(\bar g)+\I_\lo^{\partial B_{1/2}}(g-\bar g)\]
 and 
\[\I_\lo^{\partial B_{1/2}}(\bar g)-\I_\lo^{\partial B}(\bar f)\le C\|f\|^2_{L^\infty(\partial E)}\,   \dP(E).\]

Using that for $|t|\le 1$, $|\log(1+t)-t| \le \frac{t^2}{2}$, we see that
\begin{multline*}
 \int_{\partial B_{1/2}\times \partial B_{1/2}} -\log(1+\phi(x)+\phi(y)+\phi(x)\phi(y)+\psi(x,y))\, g(x) g(y)\, d\H^1(x) d\H^1(y)\\
=-\int_{\partial B_{1/2}\times \partial B_{1/2}} \left(\phi(x)+\phi(y)+\phi(x)\phi(y)+\psi(x,y)+\eta(x,y)\right)\, g(x) g(y)\, d\H^1(x) d\H^1(y)
\end{multline*}
where the function $\eta(x,y)$ is well controlled. As in Lemma \ref{shape}, 
\[\int_{\partial B_{1/2}\times \partial B_{1/2}} \phi(x)\phi(y) g(x) g(y)\, d\H^1(x) d\H^1(y)\le C
\|g\|_{L^{\infty}(\partial B_{1/2})}^2\,    \dP(E)\]
and 
\[ \int_{\partial B_{1/2}\times \partial B_{1/2}} \psi(x,y)\, g(x) g(y)\, d\H^1(x) d\H^1(y)\le C
\left(\int_{0}^{2\pi} t\ dt \right)\,  \dP(E).\]
Since
\begin{multline*}\int_{\partial B_{1/2}\times \partial B_{1/2}} \p(x) g(x) g(y) d\H^1(x) d\H^1(y)
=\bar g \int_{\partial B_{1/2}} \p(x) \left(g(x)-\bar g \right)d\H^1(x) \\
+ \bar g^2 \P(B_{1/2}) \int_{\partial B_{1/2}} \phi(x) d\H^1(x)\end{multline*}
and since $ \ds \int_{\partial B_{1/2}} \phi(x) d\H^1(x)\le C   \dP(E)$, we are left to prove that
\begin{equation}\label{lastineqlog}
\I_\lo^{\partial B_{1/2}}(g-\bar g)- \bar g \int_{\partial B_{1/2}} \p(x) \left(g(x)-\bar g \right)d\H^1(x) \ge C \bar g^2 \dP(E).
\end{equation}
As in the proof of Proposition \ref{mainstab}, we use the Fourier transform to assert that for some regular
extension
$\tilde \p$ of $\p$ and for $\widetilde{g}:= (g-\bar g)\H^1\res \partial B_{1/2}$,
\begin{align*}
 \int_{\partial B_{1/2}} \p(x) \left(g(x)-\bar g \right)d\H^1(x)&\le \left(\int_{\R^2}
 \widehat{\widetilde\vphi}^2|\xi|^2 \, d \xi\right)^{1/2} \left(\int_{\R^2} \widehat{\widetilde g}^2 |\xi|^{-2}\, d
\xi\right)^{1/2}\\
&\le C \|\p\|_{\stackrel{}{H^{1}}} \I_\lo^{\partial B_{1/2}}(g-\bar g) 
\end{align*}
 from which \eqref{lastineqlog} follows arguing exactly as in the last part of the proof
 of Proposition \ref{mainstab}.
\end{proof}
\noi Arguing as in the proof of Theorem \ref{stability}, we get the following result.

\begin{corollary}\label{stabilitylog}
Let $d=2$ then for any $\delta>0$ and $m>0$, there exists a $\bar Q\lt( \frac{\delta}{\sqrt{m}}\rt)>0$ such that, if $\frac{Q}{m^{1/4}}<\bar Q\lt( \frac{\delta}{\sqrt{m}}\rt)$, the ball is the unique 
minimizer of problem \eqref{problog} among the sets in $\Kd$ with charge $Q$.  
\end{corollary}

\begin{example}\label{acuti}\rm
In this example we show that if the boundary of a convex set is non-regular at a point $x$, then the optimal
measure for $K$ is not bounded at $x$. For simplicity we offer the example just in dimensions $2$ and $3$. It is
not difficult to extend  such an example in any dimension.
Let us start with the case $d=2$. Let $K\subset\R^2$ be a compact convex set  and let $\mu$ be the optimal measure
for $K$ in the sense of
 \eqref{eqQ}. Suppose that $x\in\partial K$ is not a regular point, that is the tangent cone of $K$ at $x$ spans
an $\gamma<\pi$. Let us denote such a cone by $C$. Up to a rotation and a translation of
 $K$ we can suppose that $x=0$ and that $C$ takes the form
 \[
  C=\{(x,y):0\ge y\ge \tan(\g)x\}.
 \]
 Let, as usual,
 $v$ be the potential of $K$ with respect to the logarithmic kernel, so that, in particular
 \[
  \begin{cases}
   -\Delta v=0\qquad &\text{on $\R^2\setminus K$}\\
   v=c &\text{on $\partial K$}.
  \end{cases}
\]
 Let us consider the function $u$ which, in polar coordinates takes the form
 \[
  u(r,\theta)=r^{\frac{\pi}{2\pi-\gamma}} \sin\left(\frac{\pi}{2\pi-\gamma} \theta\right).
 \]
 Then we can construct the barrier function $u_\eps$ as follows:
 \[
  u_\eps=c-\eps u,
 \]
where $\eps$ is a positive parameter that will be fixed later. Notice that $u_\eps$ is an harmonic function on
 $\R^2\setminus C$ which is constantly equal to $c$ on $\partial C$. 
 Since $v$ is a continuous function, we can choose a radius $R>0$ such that $v>c/2$ on $B(0,R)\cap
(\R^2\setminus C)$. By imposing
 \[
  u_\eps>v\qquad\text{on $\partial B(0,R)\cap (\R^2\setminus C) $},
 \]
 that is,
 \[
  \eps<\frac{c}{2\max_{\theta\in[0,2\pi-\gamma]}u(R,\theta)},
 \]
 we get, by the comparison principle between harmonic functions, that $u_\eps\ge v$ on $(\R^2\setminus C)\cap
B(0,R)$. Since
 $v(0)=u(0)=c$, this entails that
 
 %
   \[
     \lim_{y\to 0, y\not\in K} |\nabla v(y)|\ge|\nabla u(0)|.
   \]
 Moreover we have $|\nabla u(\rho,\theta)|=C(\gamma)\rho^{\frac{\pi}{2\pi-\gamma}-1}$
 which
 is finite in $0$ only if $\gamma\ge\pi$. We conclude thanks to Proposition
 \ref{regolarita}  that $\mu=|\nabla v|\mathcal H^{1}\res\partial
 K$ holds.
 
 To deal with the case $d=3$ we simply notice that if $\partial K$ is not regular at a point $x\in\partial K$,
where $K$ is now a convex set contained in $\R^3$, then there exist two tangent planes intersecting at $x$ which
divide $\R^3$ into two conical components of the form $C'=C\times \R$, and $\R^3\setminus C'$, being $C$
a cone of $\R^2$, and such that $K\subseteq C'$. Thus, by considering the function which in cylindric
coordinates takes the form
\[
 u(\rho,\theta,z)=r^{\frac{\pi}{2\pi-\gamma}} \sin\left(\frac{\pi}{2\pi-\gamma} \theta\right).
\]
and as before, $u_\e=c-\e u$, since such a function is harmonic in $\R^3\setminus C'$ and equals $v$ on $x$,  we
can repeat an analogous argument to that performed in the two dimensional case to show that $\infty=|\nabla
u_\e(x)|\le |\nabla v(x)|$, being $v$ the (Coulombic) potential of the set $K$.
 \end{example}

%


\section*{Acknowledgments}
We thank C. Muratov for drawing our attention to part 
of the physics litterature. M. Goldman was supported by 
a Von Humboldt post-doc fellowship. 
B. Ruffini was partially supported by the project ANR-12-BS01-
0014-01 Geometrya
and M. Novaga was partially supported by the PRIN
2010-2011 project ``Calculus of Variations''. 
The  authors wish to thanks the hospitality of the
Max Planck Institut f\"ur Mathematik in Leipzig, where this work was started.

\end{document}